\documentclass[hidelinks,onefignum,onetabnum]{siamart220329}

\usepackage[utf8]{inputenc} % to be able to add accents in the .tex file, etc.
\usepackage[T1]{fontenc} % for accents, etc., in the output document
\usepackage{amsmath,latexsym} % symbols & theorem environments
\usepackage{mathtools}
\mathtoolsset{showonlyrefs=false,mathic,centercolon} % number all equations
\usepackage{amssymb,amsfonts} % fonts & symbols
\usepackage{booktabs} % \toprule, \midrule, \bottomrule for proper table formatting
\usepackage{enumitem}
\usepackage{babel}
\usepackage{algpseudocode} % algorithmicx
\usepackage{algorithm} % put algorithmic in a floating object
\usepackage{bm} % bold symbols
\usepackage{url}
\usepackage{tikz}
\usepackage{standalone} % standalone tex file for tikz images (for reusability)
\usepackage{epstopdf}
\ifpdf
  \DeclareGraphicsExtensions{.eps,.pdf,.png,.jpg}
\else
  \DeclareGraphicsExtensions{.eps}
\fi

\usepackage[labelfont=sl,textfont=sl]{subcaption} % newer version of subfig
\usepackage[font={small,sl},labelsep=period]{caption} % format figure & table captions

\usetikzlibrary{shapes}
\usetikzlibrary{arrows.meta}

\newsiamremark{remark}{Remark}
\newsiamremark{hypothesis}{Hypothesis}
\crefname{hypothesis}{Hypothesis}{Hypotheses}
\newsiamthm{claim}{Claim}
\newsiamthm{assumption}{Assumption}

\headers{Black-box Optimization for Regularized Least-Squares Problems}{Y. Liu, K. H. Lam and L. Roberts}

\title{Black-box Optimization Algorithms for Regularized Least-Squares Problems\thanks{Submitted to the editors DATE.}}

\author{Yanjun Liu\thanks{Mathematical Sciences Institute, Building 145, Science Road, The Australian National University, Canberra ACT 2601, Australia. Current address: Department of Operations Research and Financial Engineering, Princeton University, Princeton NJ 08544, USA (\email{yanjun.liu@princeton.edu}).}
\and Kevin H.~Lam\thanks{College of Engineering, Computing and Cybernetics, The Australian National University, 108 North Rd, Acton ACT 2601, Australia (\email{Kevin.Lam@anu.edu.au}).}
\and Lindon Roberts\thanks{School of Mathematics and Statistics, Carslaw Building, University of Sydney, Camperdown NSW 2006, Australia
  (\email{lindon.roberts@sydney.edu.au}). This author was supported by the Australian Research Council Discovery Early Career Award DE240100006.}}

\usepackage{amsopn}

\newcommand{\secref}[1]{Section~\ref{#1}}

\newcommand{\thmref}[1]{Theorem~\ref{#1}}
\newcommand{\lemref}[1]{Lemma~\ref{#1}}
\newcommand{\corref}[1]{Corollary~\ref{#1}}
\renewcommand{\algref}[1]{Algorithm~\ref{#1}}
\newcommand{\assref}[1]{Assumption~\ref{#1}}

\newcommand{\R}{\mathbb{R}} % Real numbers
\newcommand{\N}{\mathbb{N}} % Natural numbers = {1,2,3,...}
\newcommand{\bigO}{\mathcal{O}} % big O notation
\DeclareMathOperator*{\argmin}{arg\,min} % argmin
\newcommand{\grad}{\nabla} % gradient
\newcommand{\prox}{\operatorname{prox}} % proximal operator
\newcommand{\norm}[1]{\left\lVert #1 \right\rVert} % \norm{x}
\newcommand{\abs}[1]{\left\lvert #1 \right\rvert} % \abs{t}
\newcommand{\kappaef}{\kappa_{\textnormal{ef}}}
\newcommand{\kappaeg}{\kappa_{\textnormal{eg}}}
\newcommand{\gammainc}{\gamma_{\textnormal{inc}}}
\newcommand{\gammadec}{\gamma_{\textnormal{dec}}}
\newcommand{\gammaincbar}{\bar{\gamma}_{\textnormal{inc}}}
\newcommand{\flow}{f_{\textnormal{low}}}
\newcommand{\hlow}{h_{\textnormal{low}}}
\newcommand{\rmax}{r_{\textnormal{max}}}
\newcommand{\jmax}{J_{\textnormal{max}}}

\def\be{\begin{align}}
\def\ee{\end{align}}

\renewcommand{\b}[1]{\bm{#1}} % bold
\newcommand{\So}{\mathcal{S}} % solver S
\newcommand{\Ps}{\mathcal{P}} % power set symbol for set of problems P

\newcommand{\bzero}{\b{0}}
\newcommand{\bx}{\b{x}}
\newcommand{\bd}{\b{d}}

\newcommand{\by}{\b{y}}
\newcommand{\bz}{\b{z}}

\newcommand{\bs}{\b{s}}

\newcommand{\br}{\b{r}}

\newcommand{\bg}{\b{g}}
\newcommand{\bem}{\b{m}}

\newcommand{\mcS}{\mathcal{S}}
\newcommand{\mcM}{\mathcal{M}}
\newcommand{\mcK}{\mathcal{K}}
\newcommand{\mcCi}{\mathcal{C}_{i_\epsilon}}
\newcommand{\mcCMi}{\mathcal{C}^{M}_{i_\epsilon}}
\newcommand{\mcCUi}{\mathcal{C}^{U}_{i_\epsilon}}
\newcommand{\mcFi}{\mathcal{F}_{i_\epsilon}}
\newcommand{\mcSi}{\mathcal{S}_{i_\epsilon}}
\newcommand{\mcMi}{\mathcal{M}_{i_\epsilon}}
\newcommand{\mcUi}{\mathcal{U}_{i_\epsilon}}

\algrenewcommand\algorithmicrequire{\textbf{Input:}}
\algrenewcommand\algorithmicensure{\textbf{Output:}}

\newcommand{\modification}[1]{{\color{black}#1}} % highlight modifications

\ifpdf
\hypersetup{
  pdftitle={Regularized black-box optimization algorithms for least-squares problems},
  pdfauthor={Y. Liu, L. Roberts}
}
\fi

\begin{document}

\maketitle

% REQUIRED
\begin{abstract}
We consider the problem of optimizing the sum of a smooth, nonconvex function for which derivatives are unavailable, and a convex, nonsmooth function with easy-to-evaluate proximal operator.
Of particular focus is the case where the smooth part has a nonlinear least-squares structure.
We adapt two existing approaches for derivative-free optimization of nonsmooth compositions of smooth functions to this setting.
Our main contribution is adapting our algorithm to handle inexactly computed stationary measures, where the inexactness is adaptively adjusted as required by the algorithm (where previous approaches assumed access to exact stationary measures, which is not realistic in this setting).
Numerically, we provide two extensions of the state-of-the-art DFO-LS solver for nonlinear least-squares problems and demonstrate their strong practical performance.
\end{abstract}

% REQUIRED
\begin{keywords}
    derivative-free optimization; nonlinear least-squares; nonsmooth optimization; trust-region methods; 
\end{keywords}

% REQUIRED
\begin{MSCcodes}
 65K05, 90C30, 90C56
\end{MSCcodes}

\section{Introduction}
Derivative-free optimization (DFO), also referred to as black-box optimization, has received growing attention for minimizing a function without accessing its derivative information. The unavailability of derivatives frequently occurs in the field of computer science and engineering with a variety of applications \cite{conn2009introduction}. There are several popular approaches to DFO methods such as direct-search and model-based methods (see \cite{larson2019derivative} for a survey). Here we focus on the derivative-free model-based algorithms based on trust-region methods (see \cite{conn2000trust, yuan2015recent}), where at each iteration a trial step is computed (inaccurately) by solving a subproblem of minimizing a model function built through interpolation within a trust region.

In this paper, we consider the minimization of a composite objective function
\begin{align}
  \min_{\bx\in \R^n} \left\{\Phi(\bx)=f(\bx)+h(\bx)\right\},  \label{eq_prob}
\end{align}
where $f: \R^n \to \R$ is smooth and potentially nonconvex and $h: \R^n \to \R \cup \{\infty\}$ is a convex but possibly nonsmooth regularization term. We assume that the derivatives of $f$ are not accessible. Furthermore, $h$ is Lipschitz continuous in $\textnormal{dom} h \coloneqq \{\bx \in \R^{n}: h(\bx)<\infty\}$ and its proximal operator is assumed to be cheap to evaluate, which might be required for solving the trust-region subproblem. 
Problems of the form \eqref{eq_prob} are widely studied in data science (e.g.~\cite{Wright2022}), and in a more traditional DFO context arise for example in image analysis \cite{Ehrhardt2021}.

Motivated by two different algorithmic approaches in \cite{garmanjani2016trust} and \cite{grapiglia2016derivative}, both primarily designed for objectives of the form $f(\bx) + h(\bm{c}(\bx))$ where both $f$ and $\bm{c}$ are black boxes, we design two algorithms for solving \eqref{eq_prob} in a derivative-free trust-region framework. 
Our ultimate aim in this work is to specialize both approaches to \eqref{eq_prob} (i.e.~$\bm{c}(\bx)=\bx$) in the specific case of regularized least-squares problems, where $f(\bx)$ has a nonlinear least-squares structure, $f(\bx) = \sum_{i} r_i(\bx)^2$ (and derivatives of $r_i$ are not available).

While the approach from \cite{garmanjani2016trust} transfers readily to this new setting, more work is required for the approach from \cite{grapiglia2016derivative}.
Specifically, we must acknowledge that the stationary measure from \cite{grapiglia2016derivative} cannot be computed exactly, despite this quantity being needed in the algorithm.
Hence, we need to extend that approach to allow for inexact estimation of the stationary measure.
Our new algorithm extending \cite{grapiglia2016derivative} also allows for  more standard assumptions on the trust-region subproblem solution and incorporates more sophisticated trust-region management strategies which have been successful in practical DFO codes \cite{Powell2009,cartis2019improving}.
We provide a full first-order convergence and worst-case complexity analysis of this, while similar guarantees from \cite{garmanjani2016trust} follow with limited extra complications.
We numerically compare both approaches and provide open-source implementations which extend the smooth least-squares package DFO-LS \cite{cartis2019improving} (see Section~\ref{sec_implementation} for details).

\subsection{Existing Works}
Derivative-based algorithms for minimizing a nonsmooth objective function are well-studied. For example, the proximal point method \cite{beck2017first} for general nonsmooth optimization, the proximal gradient method \cite{beck2017first} and its accelerated variant \cite{attouch2016rate} for minimizing convex but nonsmooth objective function in the form of \eqref{eq_prob}. The main ideas behind these proximal algorithms (see \cite{parikh2014proximal} for a comprehensive survey) consist of approximating the nonsmooth structure by a smooth function and applying efficient algorithms for smooth optimization. Based on the trust-region framework, the smoothing trust-region method proposed in \cite{chen2013optimality} uses a sequence of parameterized smoothing functions to approximate the original nonsmooth objective function, where the smoothing parameter is updated before applying the trust region method at each iteration to ensure convergence. Another approach in \cite{cartis2011evaluation} directly applies the trust-region method to solve a class of nonsmooth, nonconvex optimization problems using an appropriate criticality measure. However, the complexity analysis in \cite{cartis2011evaluation} builds on the exact evaluations of the trust region subproblem and the criticality measure, which are not required in our algorithms. 
More recently, a derivative-based trust-region method for solving \eqref{eq_prob} with $h$ nonsmooth and possibly nonconvex was introduced \cite{Aravkin2022}.

There are also several model-based DFO methods for nonconvex, nonsmooth objectives, for which the survey \cite{Audet2020a} provides a thorough discussion (and Mesh Adaptive Direct Search \cite{Audet2006} is a widely used direct search method for this setting).
For problems where the nonsmoothness arises in a structured way (such as \eqref{eq_prob}), the most common model-based DFO setting is composite objectives of the form $h(\bm{c}(\bx))$ where $h$ is nonsmooth and $\bm{c}$ is a black-box function \cite{garmanjani2016trust,grapiglia2016derivative,Hare2020d,Larson2023}.\footnote{This structure is sometimes called either `gray-box' or `glass-box' optimization, since the structure of $h$ is typically assumed to be known.}
Here, we are particularly interested in the case where $f$ has a nonlinear least-squares structure, which also has a (smooth) composite objective form.
The works \cite{zhang2010derivative, cartis2019derivative} exploit this composite structure in the least-squares case.
Additionally, \cite{conejo2013global, hough2022model} study model-based DFO methods for optimisation with convex constraints (i.e.~where $h$ is the indicator function of a convex set). 

Here, our focus is on adapting two existing approaches to the setting \eqref{eq_prob}, and we describe both these approaches below.
Firstly, \cite{grapiglia2016derivative} considers the composite objective 
\begin{align}
    \Phi(\bx) = f(\bx) + h(\bm{c}(\bx)), \label{eq_prob_composite}
\end{align}
where $f$ and $\bm{c}$ are nonconvex, derivative-free functions, and $h$ is convex but potentially nonsmooth.
Here, interpolation-based models are formed for both $f$ and $\bm{c}$ and used within a trust-region method.
Global convergence and a worst-case complexity bound of $\bigO(|\log\epsilon|\, \epsilon^{-2})$ objective evaluations to reach $\epsilon$-approximate first-order stationary are then established.
However, because of the difficulty of solving the associated subproblems (the trust-region subproblem and calculating approximate stationary estimates), it was only implemented in the case where $f=0$ and $h(\bm{c}) = \max_{i} c_i$, in which case the subproblems reduce to linear programs.

Secondly, \cite{garmanjani2016trust}, as well as establishing complexity bounds for smooth model-based DFO, considers two nonsmooth problems:
\begin{itemize}
    \item Where $\Phi(\bx)$ is nonsmooth and nonconvex with no specific structure; and
    \item The composite form \eqref{eq_prob_composite}.
\end{itemize}
In the first case, they consider the setting where $\Phi$ can be approximated by a sequence of smooth functions (parametrized by some scalar $\mu\to 0^{+}$), and introduce a method where a sequence of smooth problems with decreasing values of $\mu$ is solved inexactly (using standard smooth model-based DFO methods).
They show global convergence and a worst-case complexity bound of $\bigO(\epsilon^{-3})$ objective evaluations to achieve $\epsilon$-approximate first-order stationary.
In the composite case, they propose a very similar method to \cite{grapiglia2016derivative} but with an improved worse-case complexity bound of $\bigO(\epsilon^{-2})$ objective evaluations.

\subsection{Contributions} 

We introduce two \emph{implementable} methods based on \cite{grapiglia2016derivative} and \cite{garmanjani2016trust} for solving \eqref{eq_prob}, particularly suited for the case where $f(\bx)$ has a nonlinear least-squares structure. 

Our adaptation of \cite{grapiglia2016derivative} (as specialized to the case $\bm{c}(\bx)=\bx$ in \eqref{eq_prob_composite}) is different in several notable respects:
\begin{itemize}
    \item It allows inexact calculation of stationary measures throughout the algorithm, where the level of allowable inexactness is adapted to the algorithm's progress and requirements. When a regularization term is present, computing the stationarity measure typically requires solving a convex optimization problem, and so exact evaluations (as required by \cite{grapiglia2016derivative}) are not available, but arbitrarily accurate approximations are available through suitable iterative methods. 
    \item The sufficient decrease condition for the trust-region subproblem in \cite{grapiglia2016derivative} is proportional to the global optimality gap. Here, we propose a new, simpler sufficient decrease condition to compute the trial step, analogous to the standard Cauchy decrease condition for the smooth case (but again relying on inexactly computed stationary estimates). 
    \item The overall algorithmic framework is based on that of DFO-GN \cite{cartis2019derivative}, which has a more sophisticated trust-region mechanism which has been demonstrated to work well in practice.
\end{itemize}
We show global convergence and a worst-case complexity analysis of this approach, matching the improved $\bigO(\epsilon^{-2})$ complexity bound from \cite{garmanjani2016trust}.

Adapting the approach \cite{garmanjani2016trust} to our setting requires less effort.
We demonstrate how to derive suitable smooth approximations to \eqref{eq_prob} based on the Moreau envelope for $h$, and show that the convergence and complexity results from \cite{garmanjani2016trust} apply to this setting.
We extend the global convergence theory from \cite{garmanjani2016trust} to additionally show Clarke stationary of all accumulation points of the algorithm.

Both algorithms require the solution of specific subproblems in each iteration. 
For all subproblems encountered by both methods, we demonstrate how a smoothed variant of FISTA \cite[Chapter 10.8.4]{beck2017first} can be applied to solve these in practice.

Finally, we implement both techniques for the regularized nonlinear least-squares setting by extending the state-of-the-art solver DFO-LS \cite{cartis2019improving}, which only handles the cases where $h=0$ or $h$ is the indicator function of a convex set \cite{hough2022model}.
Both methods outperform NOMAD \cite{le2011algorithm} (which implements the mesh adaptive direct search method) on this problem class.

\paragraph{Code Availability}
Both adaptations of DFO-LS are available on Github.\footnote{From \url{https://github.com/yanjunliu-regina/dfols} and \url{https://github.com/khflam/dfols/}. See Section~\ref{sec_implementation} for details of both implementations.}

\paragraph{Structure of paper}
We first introduce our variant of \cite{grapiglia2016derivative} in Section~\ref{sec_framework}.
Our new convergence and complexity analysis of this approach is given in Section~\ref{sec_convergence}, and we then explicitly adapt this method to the nonlinear least-squares case in Section~\ref{sec_nlls}.
We then describe our smoothing variant of \cite{garmanjani2016trust} in Section~\ref{sec_alg_smoothing}.
Our approach for calculating approximate subproblem solutions is given in Section~\ref{sec_sfista}, and numerical results are provided in Section~\ref{sec_numerics}.

\paragraph{Notation}
We use $\norm{\cdot}$ to be the Euclidean norm of vectors and the operator $2$-norm of matrices, and $B(\bx, \Delta) \coloneqq\{\by \in \R^n: \norm{y-x} \leq \Delta\}$ to be the closed ball centered at $x \in \R^n$ with radius $\Delta > 0$.

\section{Algorithmic Framework} \label{sec_framework}
In this section, we outline the general model-based DFO algorithmic framework for \eqref{eq_prob}. Our framework follows the structure given in \cite{cartis2019derivative}, except the modifications for including a regularization term based on \cite{grapiglia2016derivative}. We first introduce the first-order criticality measure for \eqref{eq_prob}, and then present our algorithm. 

\subsection{Criticality Measure}
Similar to \cite{cartis2011evaluation}, we linearize $f(\bx)$ and the argument of $h$ around any $\bx$ to give an approximation of $\Phi$
\begin{align}\label{eq_cri_l}
    l(\bx,\bs) \coloneqq f(\bx)+\grad f(\bx)^{T}\bs+h(\bx+\bs), \bs \in \R^{n}.
\end{align}

Then we define the quantity
\begin{align}\label{eq_cri_psi}
    \Psi_{r}(\bx) \coloneqq l(\bx,\b0) - \min_{\norm{\bs} \leq r} l(\bx,\bs).
\end{align}

Following \cite{yuan1985conditions}, $\Psi_{r}(\bx)$ is non-negative and continuous for all $\bx \in \R^n$, and we say that $\bx^*$ is a critical point of $\Phi$ if 
\begin{align}\label{eq_cri_cond}
    \Psi_{1}(\bx^*) = 0.
\end{align}
The condition \eqref{eq_cri_cond} is equivalent to other first-order optimal necessary conditions (see \cite[Lemma 2.1]{yuan1985conditions} or the discussion in \cite[Section 2]{cartis2011evaluation}, for example).

However, in the DFO setting, we cannot calculate $\Psi_{1}$ because $\grad f$ is not accessible. 
Instead, in our algorithm we follow \cite{conn2009global} and approximate $f$ with a model constructed by interpolation to points near the current iterate. 
Let $\bx_0$ denote the initial iterate and suppose that at $k$-th iteration, we form the model through sampling within a closed ball $B(\bx_k, \Delta_k)$, where $\Delta_k$ is bounded above by $\Delta_{\max}$. 
Our notion of model accuracy is the following \cite{conn2009global}:

\begin{definition} \label{def_fl_scalar}
A model $p_k \in C^{1}$ for $f \in C^{1}$ is \emph{fully linear} in $B(\bx_k, \Delta_k)$ if for any $\by \in B(\bx_k, \Delta_k)$, 
  \begin{align}
      \abs{f(\by)-p_{k}(\by)} &\leq \kappaef \Delta_{k}^2,  \label{eq_fl_f}\\
      \norm{\grad f(\by)-\grad p_{k}(\by)} &\leq \kappaeg \Delta_{k}, \label{eq_fl_g}
  \end{align}
  where $\kappaef$ and $\kappaeg$ are independent of $\by$, $\bx_k$ and $\Delta_k$.
\end{definition} 

The local fully linear model $p_k$ (for $f$) induces an approximate criticality measure below, where in \eqref{eq_cri_l} we replace $\grad f$ with $\grad p_k$, as used in \cite{grapiglia2016derivative}.

\begin{definition}
Given a continuously differentiable model $p_k:\R^n \rightarrow \R$ for $f$, for each $\bx \in \R^n$ we define 
  \begin{align} \label{eq_cri_tildel}
    \Tilde{l}(\bx,\bs) \coloneqq f(\bx)+\grad p_k(\bx)^{T}\bs+h(\bx+\bs), \bs \in \R^{n}, 
  \end{align}
  and for all $r>0$, we define
  \begin{align} \label{eq_cri_eta}
    \eta_{r}(\bx) \coloneqq \Tilde{l}(\bx,\b0) - \min_{\norm{\bs} \leq r} \Tilde{l}(\bx,\bs).
  \end{align}
\end{definition}

We choose $\eta_{1}$ as the \emph{approximate criticality measure} for our algorithm.
When $h \equiv 0$, this $\eta_1$ reduces to $\norm{\bg_{k}}$ as expected, and when $h \equiv I_C$ is an indicator function for a convex set $C$, $\eta_1$ reduces to the approximate criticality measure chosen in \cite{hough2022model}. 

Fortunately, if $p_k$ is a fully linear model for $f$, the error between the true criticality measure $\Psi_{1}$ and our approximation $\eta_{1}$ is controlled.

\begin{lemma} \label{lem_crit_perturb}
Suppose that $f \in C^{1}$ and $h$ is continuous. Assume that $p_k: \R^n \rightarrow \R$ is a fully linear model of f with respect to constant $\kappaef$ and $\kappaeg$ on the ball $B(\bx_{k},\Delta_{k})$. Then
\begin{align} \label{eq_crit_perturb}
    \abs{\Psi_1(\by)-\eta_1(\by)} \leq \kappaeg \Delta_k,
\end{align}
for any $\by \in B(\bx_{k},\Delta_{k})$.
\end{lemma}
\begin{proof}
  This proof is based on \cite[Theorem 1]{grapiglia2016derivative}. Take any $\by \in B(\bx_{k},\Delta_{k})$. Since $p_k$ is a fully linear model of $f$ on $B(\bx_{k},\Delta_{k})$, it follows from \eqref{eq_fl_g} that 
\begin{align} \label{eq_grad_temp1}
    \norm{\grad f(\by)-\grad p_k(\by)} \leq \kappaeg \Delta_{k}.
\end{align}
Since $l(\by,\bs)$ is continuous with respect to $\bs$ on $B(\b0,1)$
, consider $\Tilde{\bs} \in B(\b0,1)$ such that
\begin{align}
    \min_{\norm{\bs} \leq 1}l(\by,\bs)=l(\by,\Tilde{\bs}).
\end{align}
Then from \eqref{eq_cri_psi}, \eqref{eq_cri_eta} and \eqref{eq_grad_temp1}, it follows that
\begin{align*} 
    \Psi_1(\by)-\eta_1(\by) &= \left(l(\by,\b0) - \min_{\norm{\bs} \leq 1} l(\by,\bs)\right) - \left(\Tilde{l}(\by,\b0) - \min_{\norm{\bs} \leq 1} \Tilde{l}(\by,\bs)\right), \\
    &= \min_{\norm{\bs} \leq 1} \Tilde{l}(\by,\bs)-\min_{\norm{\bs} \leq 1} l(\by, \bs), \\
    &= \min_{\norm{\bs} \leq 1} \Tilde{l}(\by,\bs)-l(\by,\Tilde{\bs}), \\
    &\leq \Tilde{l}(\by,\Tilde{\bs})-l(\by,\Tilde{\bs}), \\
    &= (\grad p(\by)-\grad f(\by))^{T}\Tilde{\bs}, \\
    &\leq \norm{\grad p(\by)-\grad f(\by)}, \\
    &\leq \kappaeg \Delta_k. \stepcounter{equation}\tag{\theequation}\label{eq_crit_temp1}
\end{align*}
Similarly, considering $\bar{\bs} \in B(\b0,1)$ such that 
\begin{align}
    \min_{\norm{\bs} \leq 1}\Tilde{l}(\by,\bs)=l(\by,\bar{\bs}),
\end{align}
We obtain the inequality
\begin{align} \label{eq_crit_temp2}
    \eta_1(\by)-\Psi_1(\by) \leq \kappaeg \Delta_k.
\end{align}
Hence, combining \eqref{eq_crit_temp1} and \eqref{eq_crit_temp2}, we conclude that \eqref{eq_crit_perturb} holds.
\end{proof}

\subsection{Main Algorithm}
Our algorithm for solving \eqref{eq_prob} is based on a trust-region framework \cite{conn2000trust}. 
At each iteration, we construct a model $m_k$ to approximate our objective function $\Phi$ around the current iterate $\bx_k$, and in particular within a trust region $B(\bx_k, \Delta_k)$. 
Then we find a tentative new point $\bx_k + \bs_k$ by approximately solving the trust-region subproblem
\begin{align} \label{eq_trsub}
    \bs_{k} \approx \argmin_{\norm{\bs} \leq \Delta_k} m_k(\bx_k+\bs).
\end{align}
We measure the sufficient decrease in the objective function using the ratio
\begin{align} \label{eq_ratio}
    R_k \coloneqq \frac{\Phi(\bx_k) - \Phi(\bx_k + \bs_k)}{m_k(\bx_k) - m_k(\bx_k + \bs_k)},
\end{align}
which is used to determine the next iterate $\bx_{k+1}$ and update the trust region radius $\Delta_k > 0$. If $R_k$ is sufficiently large, we accept this step (i.e. $\bx_{k+1} = \bx_{k} + \bs_{k}$) and increase $\Delta_k$, otherwise we reject this step (i.e. $\bx_{k+1} = \bx_{k}$) and decrease $\Delta_k$ if necessary.

To construct our model, we first take a quadratic approximation $p_k$ for $f$
\begin{align} \label{eq_pk}
    f(\bx_{k}+\bs) \approx p_k(\bx_{k}+\bs) = f(\bx_k)+\bg_k^{T}\bs+\frac{1}{2}\bs^{T}H_k\bs,
\end{align}
for some choices $\bg_k\in\R^n$ and $H_k\in\R^{n\times n}$. 
We can then construct our model $m_k$ for $\Phi$ as
\begin{align} \label{eq_model}
    \Phi(\bx_k + \bs) \approx m_k(\bx_{k}+\bs) \coloneqq p_k(\bx_{k}+\bs) + h(\bx_k + \bs).
\end{align}
This construction relies on our assumption that the regularizer $h$ is known and easy to evaluate.
We note the similarity between \eqref{eq_model} and the more general model used for composite nonsmooth optimization in \cite{grapiglia2016derivative}.

In practice, the model $p_k$ is constructed using the techniques from \cite{conn2009introduction}: we maintain a set of interpolation points $Y_k$ containing $\bx_k$, and approximate $\bg_k$ and $H_k$ by enforcing the interpolation conditions
\begin{align} \label{eq_interpolation}
    f(\by) = p_k(\by), \quad \forall \by\in Y_k.
\end{align}
This model construction allows $p_k$ to be fully linear provided $Y_k$ satisfies certain geometric conditions:

\begin{definition} \label{def_lambda_poised}
  Let $\{\Lambda_1(\bx), \ldots, \Lambda_{|Y_k|}(\bx)\}$ be the set of Lagrange polynomials for $Y_k$. 
  For some $\Lambda \geq 1$, $Y_k$ is \emph{$\Lambda$-poised} in some ball $B(\bx,\Delta)$ if $Y_k \subset B(\bx,\Delta)$ and
  \begin{align}
      \max_{t=1,\dots,|Y_k|}\max_{\by \in B(\bx,\Delta)}|\Lambda_t(\by)| \leq \Lambda. \label{eq_lambda_poised}
  \end{align}
\end{definition}

\begin{algorithm}[H]
\scriptsize
\begin{algorithmic}[1]
\Require Starting point $\bx_0\in\operatorname{dom} h$ 
and the initial trust-region radius $\Delta_0^{\textnormal{init}}>0$.
\vspace{0.2em}
\Statex \underline{Parameters:} maximum trust-region radius $\Delta_{\max} \geq \Delta_0^{\textnormal{init}}$ with $\Delta_{\max}>1$, trust region radius scaling factors $0<\gammadec<1<\gammainc \leq \gammaincbar$ and $0 < \alpha_1 < \alpha_2 < 1$, criticality constants $\epsilon_C,\mu>0$, acceptance thresholds $0<\beta_1\leq \beta_2<1$, trust region subproblem constant $c_1 \coloneqq \min\left\{1,\Delta_{\max}^{-2}\right\} / 2$, safety constants $0<\omega_{S}<1$ and $0<\gamma_{S}<2 e_3 c_1 / (1 + \sqrt{1+2e_3c_1})$, poisedness constant $\Lambda \geq 1$, accuracy level $0 < e_1 < 1$, $e_2 > 0$ and $0 < e_3 < 1$.
\State Build an initial interpolation set $Y_0$ of size $n+1$, with $\bx_0 \in Y_0$. Set $\rho_0^{\textnormal{init}} =\Delta_0^{\textnormal{init}}$.
\For{$k=0,1,2,\ldots$} \label{algline_forloop}
    \State Given $\bx_k$ and $Y_k$, solve the interpolation problem in \eqref{eq_interpolation} to get $J_k^{\textnormal{init}}$ and build the model $m_k^{\textnormal{init}}$ in \eqref{eq_model}. 
    \State  \label{algline_eta}Compute $\eta_1^{\textnormal{init}}(\bx_k)=\Tilde{l}^{\textnormal{init}}(\bx_k,\b0)-\Tilde{l}^{\textnormal{init}}(\bx_k,\bd_k)$, where 
    \begin{align} \label{eq_min_criticality_alg1}
        \bd_k \approx \argmin_{\norm{\bd} \leq 1} \Tilde{l}^{\textnormal{init}}(\bx_k, \bd),
    \end{align}
    and
    \begin{align}
        \Tilde{l}^{\textnormal{init}}(\bx_k,\bd)=f(\bx_k)+(\bg_k^{\textnormal{init}})^{T}\bd
        +h(\bx_k+\bd).
    \end{align}
    Here $\bd_k$ is calculated approximately so that the approximate estimate $\overline{\eta}_1^{\textnormal{init}}(\bx_k)$ satisfies
    \begin{align} \label{eq_algdfo_eta_error}
        \eta_1^{\textnormal{init}}(\bx_k)-\overline{\eta}_1^{\textnormal{init}}(\bx_k)<\min\{(1-e_1)\epsilon_C,e_2\Delta_k^{\textnormal{init}}\}
    \end{align}
    \If{$\overline{\eta}_1^{\textnormal{init}}(\bx_{k}) \leq e_1\epsilon_C$}
        \State \underline{Criticality Phase}: using \algref{alg_geo}, modify $Y_k$ and find $\Delta_k \leq \Delta_k^{\textnormal{init}}$ and $\overline{\eta}_1(\bx_k)$ such that $Y_k$ is $\Lambda$-poised in $B(\bx_k, \Delta_k)$, $\Delta_k \leq \mu \overline{\eta}_1(\bx_k)$ and $\eta_1(\bx_k) - \overline{\eta}_1(\bx_k) < e_2\Delta_k$. Set $\rho_k=\min(\rho_k^{\textnormal{init}}, \Delta_k)$. 
    \Else
        \State Set $m_k=m_k^{\textnormal{init}}$, $\Delta_k=\Delta_k^{\textnormal{init}}$, $\rho_k=\rho_k^{\textnormal{init}}$ and
        $\overline{\eta}_1(\bx_k)=\overline{\eta}_1^{\textnormal{init}}(\bx_k)$
    \EndIf
    \State \label{algline_trust_region}Approximately solve \eqref{eq_trsub} to get a step $\bs_k$ satisfying $\norm{\bs_k} \leq \Delta_k$, $m_k(\bx_{k}+\bs_k) \leq m_k(\bx_{k})$ and 
    \begin{align} \label{eq_mk_dec}
    m_k(\bx_k)-m_k(\bx_{k}+\bs_k) \geq e_3c_1\overline{\eta}_1(\bx_k)\min\left\{\Delta_k,\frac{\overline{\eta}_1(\bx_k)}{\max\bigl\{1,\norm{H_k}\bigr\}}\right\}.
\end{align}
    \State \label{algline_tau}\modification{Calculate $\tau_k \coloneqq \min\left\{\overline{\eta}_1(\bx_k) / (\norm{\bg_k}+L_h), 1\right\}$.} 
    \If{$\norm{\boldsymbol{s}_k}<  \modification{\tau_k} \gamma_S \rho_k$}\label{algline_safety}
        \State\underline{Safety Phase}: Set $\boldsymbol{x}_{k+1}=\boldsymbol{x}_k$ and $\Delta_{k+1}^{\textnormal{init}}=\max\bigl\{\rho_k, \omega_S\Delta_k\bigr\}$, and form $Y_{k+1}$ by making $Y_k$ $\Lambda$-poised in $B(\boldsymbol{x}_{k+1},\Delta_{k+1}^{\textnormal{init}})$.
        \State If $\Delta_{k+1}^{\textnormal{init}}=\rho_k$, set $(\rho_{k+1}^{\textnormal{init}},\Delta_{k+1}^{\textnormal{init}})=(\alpha_1\rho_k,\alpha_2\rho_k)$, otherwise set $\rho_{k+1}^{\textnormal{init}}=\rho_{k}$.
        \State \textbf{goto} Line \ref{algline_forloop}.
    \EndIf
    \State Evaluate $f(\bx_k+\bs_k)$ and calculate ratio $R_k$ in \eqref{eq_ratio}
    \State \label{algline_tau_delta1}Accept/reject step and update trust region radius: set
    \begin{align*}
        \bx_{k+1}=\begin{cases} 
            \bx_k+\bs_k & R_k \geq \beta_1, \\
            \bx_k &  R_k < \beta_1,
        \end{cases}
        \quad \text{and  } 
        \Delta_{k+1}^{\textnormal{init}}=\begin{cases} 
            \min(\max(\gammainc\Delta_k, \gammaincbar\norm{\bs_k}),\Delta_{\max}), & R_k \geq \beta_2, \\
            \max(\gammadec\Delta_k,\norm{\bs_k},\rho_k), &  \beta_1 \leq R_k < \beta_2, \\
        \max(\min(\gammadec\Delta_k,\norm{\bs_k})/\modification{\tau_k},\rho_k), & R_k < \beta_1.
    \end{cases}
\end{align*}
\If{$R_k \geq \beta_1$}
\State Form $Y_{k+1}=Y_{k} \cup \{\boldsymbol{x}_{k+1}\} \setminus \{\boldsymbol{y}_t\}$ for some $\boldsymbol{y}_t \in Y_{k}$ and set $\rho_{k+1}^{\textnormal{init}}=\rho_k$.
\ElsIf{$Y_k$ is not $\Lambda$-poised in $B(\boldsymbol{x}_k,\Delta_k)$}
\State \underline{Model Improvement Phase}: Form $Y_{k+1}$ by making $Y_k$ $\Lambda$-poised in $B(\boldsymbol{x}_{k+1},\Delta_{k+1}^{\textnormal{init}})$ and set $\rho_{k+1}^{\textnormal{init}}=\rho_k$.
\Else{  $[R_k<\beta_1$ \textbf{and} $Y_k$ is $\Lambda$-poised in $B(\boldsymbol{x}_k,\Delta_k)]$}
\State \label{algline_tau_delta2}\underline{Unsuccessful Phase}: Set $Y_{k+1}=Y_{k}$, and if 
$\Delta_{k+1}^{\textnormal{init}}=\rho_{k}$, set $(\rho_{k+1}^{\textnormal{init}},\Delta_{k+1}^{\textnormal{init}})=(\alpha_1\rho_k,\alpha_2\rho_k)$, otherwise set $\rho_{k+1}^{\textnormal{init}}=\rho_k$.
\EndIf
\EndFor
\end{algorithmic}
\caption{DFO-LSR: a model-based DFO method for \eqref{eq_prob}.}
\label{alg_dfo}
\end{algorithm}

Intuitively, the set $Y_k$ has `good geometry' if $Y_k$ is $\Lambda$-poised with a small $\Lambda$. 
The exact form of interpolation used to construct $p_k$ is left open for now, and will be revisited in the context of regularized nonlinear least-squares problems in \secref{sec_nlls}.
For now, in light of the techniques given in \cite[Chapter 6]{conn2009introduction} we simply assume the existence of procedures which can verify whether or not a set $Y_k$ is $\Lambda$-poised, and if it is not, add and remove points from $Y_k$ until it becomes $\Lambda$-poised.

\begin{remark}
	If $p_k$ is a linear model, the techniques from \cite[Section 4]{hough2022model} allow Definition~\ref{def_lambda_poised} to be weakened to only maximizing over $\by\in B(\bx,\Delta)\cap C$ in \eqref{eq_lambda_poised} for any closed convex set $C$ with nonempty interior.
	In our case, we might take $C=\operatorname{dom} h$, for example.
\end{remark}

Our main algorithm for solving \eqref{eq_prob} is presented in \algref{alg_dfo}.
The overall structure is based on DFO-GN \cite[Algorithm 1]{cartis2019derivative}, which is designed for nonlinear least-squares minimization (c.f.~\secref{sec_nlls}). 
Motivated by the practical performance of DFO-GN, we keep important algorithmic features not present in other similar methods (e.g.~\cite[Chapter 10]{conn2009introduction}) such as the safety phase and the maintenance of a lower bound $\rho_k$ on the trust region radius $\Delta_k$. 

In particular, we extend the safety phase from DFO-GN (which originates in \cite{powell2009bobyqa}), which provides a way to detect insufficient decrease generated by the step size $\norm{\bs_k}$ before evaluating $f(\bx_k)$ and $R_k$ \eqref{eq_ratio}. 
Specifically, to accommodate the proof of \lemref{lem_step_bound}, we introduce a new variable $\tau_k$ in the generalized safety phase. 
This $\tau_k$ is calculated in line~\ref{algline_tau} and then used in the entry condition (line~\ref{algline_safety}) and the update rule for $\Delta_k$  (lines~\ref{algline_tau_delta1} and~\ref{algline_tau_delta2}). 
If $h \equiv 0$, then $\tau_k = 1$ and the generalised safety phase is equivalent to the original safety phase from DFO-GN. 

Of course, we also use our new criticality measure $\eta_1(\bx_k)$ instead of $\norm{\bg_k}$ in DFO-GN algorithm, just as in \cite{grapiglia2016derivative}. 
Unfortunately, to calculate the value of $\eta_{1}(\bx_k)$, we need to solve a minimization subproblem as defined in \eqref{eq_cri_eta}. 
So to make our framework practical, we therefore extend \cite{grapiglia2016derivative} to allow an approximate estimate of $\eta_{1}(\bx_k)$ up to a predetermined accuracy level (line~\ref{algline_eta}). 
Note that if we compute $\bd_k$ in \eqref{eq_min_criticality_alg1} inexactly, then we automatically have $\overline{\eta}_1(\bx_k) \leq \eta_1(\bx_k)$.

With the nonsmooth term in our model $m_k$ \eqref{eq_model}, our trust-region subproblem \eqref{eq_trsub} and criticality estimation subproblem \eqref{eq_min_criticality_alg1} are not straightforward to solve to the required accuracy.
We discuss this issue further in Sections~\ref{sec_nlls} and \ref{sec_sfista}.

We describe the geometry-improvement step used in the criticality phase below, which is an adaptation of \cite[Algorithm 2]{cartis2019derivative}. 
Compared to \algref{alg_dfo}, the criticality measure is evaluated approximately up to a different accuracy in line~\ref{alg2line_eta}.

\begin{algorithm}[t]
\scriptsize
\begin{algorithmic}[1]
\Require Iterate $\bx_k$, initial set $Y_k$ and trust region radius $\Delta_k^{\textnormal{init}}$.
\vspace{0.2em}
\Statex \underline{Parameters:} $\mu>0$, $\omega_C \in (0,1)$ and poisedness constant $\Lambda>0$.
\State Set $Y_k^{(0)}=Y_{k}$.
\For{i=1, 2, ...}
    \State Form $Y_k^{(i)}$ by modifying $Y_k^{(i-1)}$ until it is $\Lambda$-poised  in $B(\bx_k,\omega_C^{i-1}\Delta_k^{\textnormal{init}})$.
    \State Solve the interpolation system for $Y_k^{(i)}$ to get $J_k^{(i)}$, and form $m_k^{(i)}$ in \eqref{eq_model}.
    \State  \label{alg2line_eta}Compute $\eta_1^{(i)}(\bx_k)=\Tilde{l}^{(i)}(\bx_k,\b0)-\Tilde{l}^{(i)}(\bx_k,\Tilde{\bd}_k)$, where 
    \begin{align} \label{eq_min_criticality_alg2}
        \Tilde{\bd}_k \approx \argmin_{\norm{\bd} \leq 1} \Tilde{l}^{(i)}(\bx_k, \bd),
    \end{align}
    and 
    \begin{align}
        \Tilde{l}^{(i)}(\bx_k,\bd)=f(\bx_k)+(\bg_k^{(i)})^{T}\bd + h(\bx_k+\bd).
    \end{align}
    Here $\Tilde{\bd}_k$ is calculated approximately so that the approximate estimate $\overline{\eta}_1^{(i)}(\bx_k)$ satisfies 
    \begin{align} \label{eq_alggeo_eta_error}
        \eta_1^{(i)}(\bx_k)-\overline{\eta}_1^{(i)}(\bx_k)<e_2\omega_C^{i-1}\Delta_k^{\textnormal{init}}
    \end{align}
    \If{$\omega_C^{i-1}\Delta_k^{\textnormal{init}}\leq \mu\overline{\eta}_1^{(i)}(\boldsymbol{x}_k)$}
        \State\Return $Y_k^{(i)}$, $m_k^{(i)}$, $\Delta_k \leftarrow \omega_{C}^{i-1}\Delta_k^{\textnormal{init}}$, $\overline{\eta}_1(\boldsymbol{x}_k) \leftarrow \overline{\eta}_1^{(i)}(\boldsymbol{x}_k)$. 
    \EndIf
\EndFor
\end{algorithmic}
\caption{Geometry-Improvement for Criticality Phase}
\label{alg_geo}
\end{algorithm}

\section{Convergence and Complexity Analysis} \label{sec_convergence}
Throughout this section, we consider the following standard assumptions:

\begin{assumption} \label{ass_smooth}
	The function $f$ \eqref{eq_prob} is bounded below by $\flow$ and  continuously differentiable in the convex hull of  $\mathcal{B} \coloneqq \cup_k B(\boldsymbol{x}_k,\Delta_{\max})$, where $\grad f$ is $L_{\grad f}$-Lipschitz continuous in $\mathcal{B}$.
\end{assumption}

\begin{assumption} \label{ass_regu}
  $h: \mathbb{R}^{n} \rightarrow \mathbb{R}$ is convex and possibly nonsmooth. We also assume that $h(\boldsymbol{x})$ is bounded below by $\hlow$ and Lipschitz continuous with Lipschitz constant $L_{h}$ in $\operatorname{dom} h$.
\end{assumption}

We also need the assumption that the model Hessians are uniformly bounded above for the trust-region subproblem \eqref{eq_trsub}.
\begin{assumption} \label{ass_boundedhess}
  There exists $\kappa_H \geq 1$ such that $\|H_k\| \leq \kappa_H - 1$ for all $k$.
\end{assumption}

Lastly, we use the below result to link the $\Lambda$-poisedness of $Y_k$ with the accuracy of the model $p_k$.

\begin{lemma} \label{lem_kappa_generic}
	If \assref{ass_smooth} holds and $Y_k$ is $\Lambda$-poised in $B(\bx_k,\Delta_k)$, then $p_k$ is a fully linear model for $f$.
\end{lemma}
\begin{proof}
	Different versions of this result are applicable depending on the specific model construction used for $p_k$, e.g.~\cite[Theorems 2.11 \& 2.12 or Theorem 5.4]{conn2009introduction}. When we specialize to nonlinear least-squares problems in \secref{sec_nlls}, this will be given by \lemref{lem_kappa}.
\end{proof}

\subsection{Global Convergence Analysis}

The following lemma ensures that unless the current iterate is a critical point, \algref{alg_geo} for the criticality phase terminates in finite loops. It also provides a bound for the trust region radius $\Delta_k$.

\begin{lemma} \label{lem_delta_bound}
Suppose Assumptions~\ref{ass_smooth} and \ref{ass_regu} hold and $\Psi_1(\bx_k) \geq \epsilon > 0$. Then for any $\mu >0$ and $\omega_C \in (0,1)$, the criticality phase in \algref{alg_geo} terminates in finite time with $Y_k$ $\Lambda$-poised in $B(\bx_k,\Delta_k)$ and $\Delta_k \leq \mu \overline{\eta}_1(\bx_k)$ for any $\mu>0$ and $\eta_1(\bx_k) - \overline{\eta}_1(\bx_k) < e_2\Delta_k$. 
We also have the bound
\begin{align} \label{eq_delta_bound}
    \min\left(\Delta_k^{\textnormal{init}},\frac{\omega_C\epsilon}{\kappaeg+1/\mu+e_2}\right) \leq \Delta_k \leq \Delta_k^{\textnormal{init}}.
\end{align}
\end{lemma}
\begin{proof}
This proof is similar to \cite[Lemma B.1]{cartis2019derivative} except that we allow the inaccurate estimate of our criticality measure $\eta_1$. First, suppose \algref{alg_geo} terminates on the first iteration. Then $\Delta_k=\Delta_k^{\textnormal{init}}$, and the result holds. Otherwise, consider some iterations $i$ where \algref{alg_geo} does not terminate, that is, where $\omega_C^{i-1}\Delta_k^{\textnormal{init}} > \mu \overline{\eta}_1^{(i)}(\bx_k)$. Then since $m_k^{(i)}$ is fully linear in $B(\bx_k, \omega_C^{i-1}\Delta_k^{\textnormal{init}})$, from \lemref{lem_crit_perturb} and \eqref{eq_alggeo_eta_error}, we have
\begin{align}
    \epsilon \leq \Psi_1(\bx_k) &= \Psi_1(\bx_k)-\eta_1^{(i)}(\bx_k)+\eta_1^{(i)}(\bx_k) \\
    &\leq
    \kappaeg\omega_C^{i-1}\Delta_k^{\textnormal{init}}+\overline{\eta}_1^{(i)}(\bx_k)+e_2\omega_C^{i-1}\Delta_k^{\textnormal{init}} \\
    &\leq \left(\kappaeg+\frac{1}{\mu}+e_2\right)\omega_C^{i-1}\Delta_k^{\textnormal{init}},
\end{align}
or equivalently $\omega_C^{i-1} \geq \dfrac{\epsilon}{(\kappaeg+1/\mu+e_2)\Delta_k^{\textnormal{init}}}$. That is, if termination does not occur on iteration $i$, we must have
\begin{align}
    i \leq 1+\frac{1}{|\log\omega_C|}\log\left(\frac{(\kappaeg+1/\mu+e_2)\Delta_k^{\textnormal{init}}}{\epsilon}\right),
\end{align}
so \algref{alg_geo} terminates finitely. We also have $\omega_C^{i-1}\Delta_k^{\textnormal{init}} \geq \dfrac{\epsilon}{\kappaeg+1/\mu+e_2}$, which gives \eqref{eq_delta_bound}.
\end{proof}

Next, we show that having a fully linear model $p_k$ for $f$ guarantees a condition similar to \eqref{eq_fl_f} for our full (nonsmooth) model $m_k$ \eqref{eq_model}.

\begin{lemma} \label{lem_model_obj_diff}
Suppose Assumptions~\ref{ass_smooth}, \ref{ass_regu} and \ref{ass_boundedhess} hold. 
If $p_k$ is a fully linear model of $f$ on the ball $B(\bx_k,\Delta_k)$, then
\begin{align}
    \abs{\Phi(\bx_k+\bs_k)-m_k(\bx_k+\bs_k)} \leq \kappaef\Delta_k^{2}.
\end{align}
\end{lemma}
\begin{proof}
From \eqref{eq_mk_dec}, $\bs_k$ must be calculated such that $\bx_k+\bs_k\in\operatorname{dom} h$.
Since $\bx_0\in\operatorname{dom} h$ by definition, we must have $\bx_k\in\operatorname{dom} h$ for all $k$.
Hence by definition of $\Phi$ and $m_k$ we have
\begin{align}
	\Phi(\bx_k+\bs_k)-m_k(\bx_k+\bs_k) &= f(\bx_k+\bs_k)-p_k(\bx_k+\bs_k)
\end{align}
and the result follows from \eqref{eq_fl_f}.
\end{proof}

In \algref{alg_dfo}, if the safety phase is not called, we say an iteration is `successful' if $R_k \geq \beta_1$ and `very successful' if $R_k \geq \beta_2$. 

\begin{lemma} \label{lem_suc_safety}
Suppose Assumptions~\ref{ass_smooth}, \ref{ass_regu} and \ref{ass_boundedhess} hold. 
If $p_k$ is a fully linear model of $f$ on the ball $B(\bx_k,\Delta_k)$, and 
\begin{align} \label{eq_suc_safety}
    \Delta_k \leq \min\left\{\frac{1}{\kappa_H},\frac{e_3c_1(1-\beta_2)}{\kappaef}\right\}\overline{\eta}_1(\bx_k),
\end{align}
then, the $k$-th iteration is either very successful or the safety phase is called.
\end{lemma}
\begin{proof}
This proof is based on \cite[Lemma 5.3]{conn2009global}.
Since $\norm{H_k}\leq\kappa_H$ and $\Delta_k\leq\dfrac{\overline{\eta}_1(\bx_k)}{\kappa_H}$, it follows from \lemref{lem_mk_dec} that
\begin{align*}
    m_k(\bx_k)-m_k(\bx_k+\bs_k) &\geq e_3c_1\overline{\eta}_1(\bx_k)\min\left\{\Delta_k,\frac{\overline{\eta}_1(\bx_k)}{\max\left\{1,\norm{H_k}\right\}}\right\} \\
    &\geq e_3c_1\overline{\eta}_1(\bx_k)\min\left\{\Delta_k,\frac{\overline{\eta}_1(\bx_k)}{\kappa_H}\right\} \\
    &= e_3c_1\overline{\eta}_1(\bx_k)\Delta_k.
\end{align*}
Then, applying \lemref{lem_model_obj_diff}, we obtain
\begin{align*}
    \abs{1-R_k} &= \abs{\frac{(m_k(\bx_k)-m_k(\bx_k+\bs_k))-(\Phi(\bx_k)-\Phi(\bx_k+\bs_k))}{m_k(\bx_k)-m_k(\bx_k+\bs_k)}} \\
    &= \frac{\abs{\Phi(\bx_k+\bs_k)-m_k(\bx_k+\bs_k)}}{\abs{m_k(\bx_k)-m_k(\bx_k+\bs_k)}} \\
    &\leq \frac{\kappaef\Delta_k}{e_3c_1\overline{\eta}_1(\bx_k)} \\
    &\leq 1-\beta_2,
\end{align*}
where we use the fact that $\Delta_k \leq \dfrac{e_3c_1(1-\beta_2) \overline{\eta}_1(\bx_k)}{\kappaef}$. Hence, $R_k \geq \beta_2$, and by \algref{alg_dfo}, the safety phase is called if $\norm{\bs_k} \leq \tau_k\gamma_S \rho_k$, otherwise the iteration $k$ is very successful.
\end{proof} 

The next lemma provides a lower bound for the trust region step size $\norm{\bs_k}$, which will be used later to determine that the safety phase is not called under appropriate assumptions.
It is based on \cite[Lemma 3.6]{cartis2019derivative} but with more work required to handle the non-standard assumption on the trust-region subproblem decrease \eqref{eq_mk_dec}.

\begin{lemma} \label{lem_step_bound}
Suppose Assumptions~\ref{ass_smooth} and \ref{ass_regu} hold. 
Then the step $\bs_k$ satisfies 
\begin{align} \label{eq_step_bound}
    \norm{\bs_k} \geq c_2\tau_k\min\left\{\Delta_k,\frac{\overline{\eta}_1(\bx_k)}{\max\left\{1,\norm{H_k}\right\}}\right\},
\end{align}
where $c_2 \coloneqq \dfrac{2 e_3 c_1}{1+\sqrt{1+2e_3 c_1}}<1$ and $\tau_k \coloneqq \min\left\{\dfrac{\overline{\eta}_1(\bx_k)}{\norm{\bg_k} + L_h},1\right\}$.
\end{lemma}
\begin{proof}
Let $h_k \coloneqq \max\{\norm{H_k},1\}\geq1$. Since $m_k(\bx_k)-m_k(\bx_k + \bs_k)\geq0$ by \algref{alg_dfo}, it follows that
\begin{align}
    m_k(\bx_k)-m_k(\bx_k + \bs_k) &= \abs{m_k(\bx_k)-m_k(\bx_k+\bs_k)}\\
    &= \abs{\bg_k^{T}\bs_{k}+\frac{1}{2}\bs_{k}^{T}H_k\bs_{k}-h(\bx_k)+h(\bx_k+\bs_k)} \\
    &\leq \abs{\bg_k^{T}\bs_{k}+\frac{1}{2}\bs_{k}^{T}H_k\bs_{k}}+\abs{h(\bx_k)-h(\bx_k+\bs_k)} \\
    &\leq \norm{\bg_k}\cdot\norm{\bs_k}+\frac{h_k}{2}\norm{\bs_k}^{2}+L_h\norm{\bs_k}.
\end{align}
Substituting this into \eqref{eq_mk_dec}, we have
\begin{align} \label{eq_quadratic_temp1}
    \frac{1}{2}\norm{\bs_k}^{2}+\frac{\norm{\bg_k}+L_h}{h_k}\norm{\bs_k}-e_3c_1\frac{\overline{\eta}_1(\bx_k)}{h_k}\min\left\{\Delta_k,\frac{\overline{\eta}_1(\bx_k)}{h_k}\right\}\geq0.
\end{align}
Define $C_k^{g} \coloneqq (\norm{\bg_k} + L_h) / h_k$ and $C_k^{\eta} \coloneqq \overline{\eta}_1(\bx_k) / h_k$. For \eqref{eq_quadratic_temp1} to be satisfied, we require that $\norm{\bs_k}$ is no less than the positive root of the left-hand side of \eqref{eq_quadratic_temp1}, which gives the inequality below
\begin{align} \label{eq_quadratic_temp2}
    \norm{\bs_k} &\geq \frac{2e_3c_1C_k^{\eta}}{C_k^{g}+\sqrt{(C_k^{g})^2+2e_3c_1(C_k^{\eta})^2}} \min\left\{\Delta_k,C_k^{\eta}\right\}
\end{align}
Note that $\tau_k = \min\{C_k^{\eta} / C_k^{g}, 1\}$. If $\tau_k \geq 1$, i.e., $C_k^{\eta} \geq C_k^{g}$, from \eqref{eq_quadratic_temp2},
\begin{align*} 
    \norm{\bs_k} &\geq \frac{2e_3c_1C_k^{\eta}}{C_k^{\eta}+\sqrt{(C_k^{\eta})^2+2e_3c_1 (C_k^{\eta})^2}} \min\left\{\Delta_k,C_{k}^{\eta}\right\} \\
    &= \frac{2e_3 c_1}{1 + \sqrt{1+2e_3c_1}} \min\left\{\Delta_k, C_k^{\eta}\right\}
    \stepcounter{equation}\tag{\theequation}\label{eq_quadratic_temp3}
\end{align*}
On the other hand, if $\tau_k < 1$, that is, $C_k^{\eta} < C_k^{g}$, it follows from \eqref{eq_quadratic_temp2} that 
\begin{align*} 
    \norm{\bs_k} &\geq \frac{2e_3c_1 C_k^{g} \left(C_k^{\eta} / C_k^{g}\right)}{C_k^{g} + \sqrt{(C_k^{g})^2 + 2e_3c_1(C_k^{g})^2}} \min\left\{\Delta_k, C_k^{\eta}\right\} \\
    &= \frac{2e_3 c_1 \left(C_k^{\eta} / C_k^{g}\right)}{1 + \sqrt{1+2e_3c_1}} \min\left\{\Delta_k, C_k^{\eta}\right\}.
    \stepcounter{equation}\tag{\theequation}\label{eq_quadratic_temp4}
\end{align*}
Combining \eqref{eq_quadratic_temp3} and \eqref{eq_quadratic_temp4}, 
\begin{align}
    \norm{\bs_k} &\geq \frac{2e_3 c_1 \tau_k}{1 + \sqrt{1+2e_3c_1}}\min\left\{\Delta_k, C_k^{\eta}\right\},
\end{align}
Then we can recover \eqref{eq_step_bound} from the definition of $h_k$, $C_k^{\eta}$ and $c_2$.
We conclude by noting $c_2<1$ since $e_3<1$ and $c_1<1/2$.
\end{proof}

\begin{lemma} \label{lem_eta_bound}
In all iterations, $\eta_1(\bx_k) \geq \overline{\eta}_1(\bx_k) \geq \min\left\{e_1\epsilon_C, \dfrac{\Delta_k}{\mu}\right\}$. Also, if $\Psi_1(\bx_k)\geq\epsilon>0$, then 
\begin{align} \label{eq_eta_bound}
    \eta_1(\bx_k) \geq \overline{\eta}_1(\bx_k)\geq\epsilon_g\coloneqq\min\left\{e_1\epsilon_C, \frac{\epsilon}{1+(\kappaeg+e_2)\mu}\right\}>0.
\end{align}
\end{lemma}
\begin{proof}
Firstly, since $\eta_1(\bx_k) \geq \overline{\eta}_1(\bx_k)$, if the criticality phase is not called, then we must have $\eta_1(\bx_k) \geq \overline{\eta}_1(\bx_k)=\overline{\eta}_1^{\textnormal{init}}(\bx_k)>e_1\epsilon_C$. Otherwise, we have $\Delta_k \leq \mu\overline{\eta}_1(\bx_k) \leq \mu\eta_1(\bx_k)$. Hence, $\eta_1(\bx_k) \geq \overline{\eta}_1(\bx_k) \geq \min\left\{e_1\epsilon_C, \dfrac{\Delta_k}{\mu}\right\}$. 

The proof of \eqref{eq_eta_bound} follows from \cite[Lemma 10.11]{conn2009introduction} except that we must take into account the inaccuracy when computing $\overline{\eta}_1(\bx_k)$. We first suppose that the criticality phase is not called. Then  $\overline{\eta}_1(\boldsymbol{x}_k)=\overline{\eta}_1^{\textnormal{init}}(\boldsymbol{x}_k)>e_1\epsilon_C$ and hence \eqref{eq_eta_bound} holds. Otherwise, the criticality phase is called and $m_k$ is fully linear in $B(\bx_k,\Delta_k)$ with $\Delta_k \leq \mu\overline{\eta}_1(\bx_k)$. In this case, applying \lemref{lem_crit_perturb}, we obtain
\begin{align}
    \Psi_1(\bx_k) &=\Psi_1(\bx_k)-\eta_1(\bx_k)+\eta_1(\bx_k) \\
    &\leq \kappaeg\Delta_k+\overline{\eta}_1(\bx_k)+e_2\Delta_k \\
    &\leq (\kappaeg+e_2)\mu\overline{\eta}_1(\bx_k)+\overline{\eta}_1(\bx_k).
\end{align}
Since $\Psi_1(\bx_k) \geq \epsilon$, then $\overline{\eta}_1(\bx_k) \geq \dfrac{\epsilon}{1+(\kappaeg+e_2)\mu}$ and \eqref{eq_eta_bound} holds.
\end{proof}

\begin{lemma} \label{lem_rho_min}
Suppose that \assref{ass_smooth}, \assref{ass_regu} and \assref{ass_boundedhess} hold. If $\Psi_1(\bx_k) \geq \epsilon>0$ for all $k$, then $\rho_k\geq\rho_{\min}>0$ for all $k$, where
{\small
\begin{align} \label{eq_rho_min}
    \rho_{\min} \coloneqq \min\left\{\Delta_0^{\textnormal{init}},\frac{\omega_C\epsilon}{\kappaeg+1/\mu+e_2},\frac{\modification{c_2 \gammadec}\alpha_1\epsilon_g}{\modification{2}\kappa_H},\frac{\modification{c_2 \gammadec}\alpha_1}{\modification{2}}\left(\kappaeg+\frac{\kappaef}{e_3c_1(1-\beta_2)}+e_2\right)^{-1}\epsilon\right\}.
\end{align}
}
\end{lemma}
\begin{proof}
This proof is similar to \cite[Lemma 3.8]{cartis2019derivative}. To find a contradiction, suppose that $k(0)$ is the first $k$ such that $\rho_k<\rho_{\min}$. That is, we have
\begin{align} \label{eq_rho_temp1}
    \rho_0^{\textnormal{init}} \geq \rho_0 \geq \rho_1^{\textnormal{init}} \geq \rho_1 \geq \dots \geq \rho_{k(0)-1}^{\textnormal{init}} \geq \rho_{k(0)-1} \geq \rho_{\min} \quad \text{and} \quad \rho_{k(0)}<\rho_{\min}.
\end{align}
We first claim that 
\begin{align} \label{eq_rho_temp2}
    \rho_{k(0)}=\rho_{k(0)}^{\textnormal{init}}<\rho_{\min}.
\end{align} 
To see this, note that we have either $\rho_{k(0)}=\rho_{k(0)}^{\textnormal{init}}$ or $\rho_{k(0)}=\Delta_{k(0)}$ from \algref{alg_dfo}. In the former case, \eqref{eq_rho_temp2} holds trivially. In the latter, using \lemref{lem_delta_bound} and \eqref{eq_rho_temp1}, we have that 
\begin{align}
    \rho_{\min}>\Delta_{k(0)} \geq \min\left(\Delta_{k(0)}^{\textnormal{init}},\frac{\omega_C\epsilon}{\kappaeg+1/\mu+e_2}\right)
    \geq \min\left(\rho_{k(0)}^{\textnormal{init}},\frac{\omega_C\epsilon}{\kappaeg+1/\mu+e_2}\right).
\end{align}
Since $\rho_{\min} \leq \dfrac{\omega_C\epsilon}{\kappaeg+1/\mu+e_2}$, then we conclude that \eqref{eq_rho_temp2} holds. 

Since $\rho_{\min} \leq \Delta_0^{\textnormal{init}}=\rho_0^{\textnormal{init}}$ by \algref{alg_dfo}, we must have $k(0)>0$ and $\rho_{k(0)-1} \geq \rho_{\min}>\rho_{k(0)}^{\textnormal{init}}$. This reduction in $\rho$ can only happen from a safety step or an unsuccessful step, and in both cases, we have $\rho_{k(0)}^{\textnormal{init}}=\alpha_1\rho_{k(0)-1}$. Therefore, $\rho_{k(0)-1} \leq \rho_{\min} /\alpha_1$. Also, by \lemref{lem_step_bound}, \lemref{lem_eta_bound} and \assref{ass_boundedhess}, we have 
\begin{align} \label{eq_rho_temp3}
     \norm{\bs_{k(0)-1}} \geq c_2 \tau_{k(0)-1} \min\left\{\Delta_{k(0)-1}, \frac{\epsilon_g}{\kappa_H}\right\}.
\end{align}

If we have a safety step, we know 
\begin{align} \label{eq_rho_temp4}
    \norm{\bs_{k(0)-1}} \leq \tau_{k(0)-1}\gamma_S\rho_{k(0)-1} \leq \frac{\tau_{k(0)-1}\gamma_S}{\alpha_1}\rho_{\min}.
\end{align}
Combining \eqref{eq_rho_temp3} and \eqref{eq_rho_temp4}, 
it follows from the assumption $\gamma_S < c_2$ in \algref{alg_dfo} that
\begin{align}
    \min\left\{\Delta_{k(0)-1}, \frac{\epsilon_g}{\kappa_H}\right\} \leq \frac{\gamma_S \rho_{\min}}{c_2 \alpha_1} < \frac{\rho_{\min}}{\alpha_1}.
\end{align}
Since $\rho_{\min} / \alpha_1 \leq \modification{c_2 \gammadec \epsilon_g / (2\kappa_H)} \leq \epsilon_g / \kappa_H$, then  $\min\left\{\Delta_{k(0)-1}, \epsilon_g / \kappa_H\right\} = \Delta_{k(0)-1}$. Hence in \eqref{eq_rho_temp3}, $\norm{\bs_{k(0)-1}} \geq c_2 \tau_{k(0)-1} \Delta_{k(0)-1} > \tau_{k(0)-1} \gamma_S\rho_{k(0)-1}$
, which contradicts our assumption that the safety step is called. Therefore, the iteration $k(0)-1$ must be an unsuccessful step. To obtain reduction in $\rho$, the update rule of $\Delta_{k(0)}^{\textnormal{init}}$ in \algref{alg_dfo} implies that 
\begin{align}
    \gammadec \norm{\bs_{k(0)-1}} \leq \min\left\{\gammadec \Delta_{k(0)-1}, \norm{\bs_{k(0)-1}}\right\} \leq \tau_{k(0)-1} \rho_{k(0)-1}.
\end{align}
Therefore, $\norm{\bs_{k(0)-1}} \leq \tau_{k(0)-1} \gammadec^{-1} \rho_{k(0)-1} \leq \tau_{k(0)-1} \gammadec^{-1} \alpha_1^{-1} \rho_{\min}$. From this, \eqref{eq_rho_temp3} together with \eqref{eq_rho_min}, we have 
\begin{align} \label{eq_rho_temp4.5}
    \min\left\{\Delta_{k(0)-1}, \frac{\epsilon_g}{\kappa_H}\right\} &\leq \frac{\rho_{\min}}{c_2 \gammadec \alpha_1} \notag\\
    &< \frac{2\rho_{\min}}{c_2 \gammadec \alpha_1} \notag\\ &\leq \min\left\{\frac{\epsilon_g}{\kappa_H}, \left(\kappaeg + \frac{\kappaef}{e_3 c_1 (1-\beta_2)}+e_2\right)^{-1}\epsilon\right\}.
\end{align}
Then $\min\left\{\Delta_{k(0)-1}, \epsilon_g/\kappa_H\right\} = \Delta_{k(0)-1}$. 
From this and \lemref{lem_eta_bound}, we obtain
\begin{align} \label{eq_rho_temp5}
    \Delta_{k(0)-1} \leq \left(\kappaeg + \frac{\kappaef}{e_3 c_1 (1-\beta_2)}+e_2\right)^{-1}\epsilon \quad \text{and} \quad \Delta_{k(0)-1} \leq \frac{\epsilon_g}{\kappa_H} \leq \frac{\overline{\eta}_1(\bx_{k(0)-1})}{\kappa_H}.
\end{align}

Now suppose that 
\begin{align} \label{eq_rho_temp6}
    \Delta_{k(0)-1} > \dfrac{e_3c_1(1-\beta_2)}{\kappaef}\overline{\eta}_1(\bx_{k(0)-1}).
\end{align} From \algref{alg_dfo}, regardless of the call of the criticality phase, we have 
\begin{align}
    \eta_1(\bx_{k(0)-1})-\overline{\eta}_1(\bx_{k(0)-1})<e_2\Delta_{k(0)-1}.
\end{align}
Then using \lemref{lem_crit_perturb} and \eqref{eq_rho_temp6} we have
\begin{align}
    \epsilon \leq \Psi_1(\boldsymbol{x}_{k(0)-1}) &= \Psi_1(\boldsymbol{x}_{k(0)-1}) - \eta_1(\bx_{k(0)-1}) + \eta_1(\bx_{k(0)-1}) \\
    &\leq \kappaeg\Delta_{k(0)-1}+\eta_1(\bx_{k(0)-1}) \\
    &\leq (\kappaeg+e_2)\Delta_{k(0)-1}+\overline{\eta}_1(\bx_{k(0)-1}) \\
    &< \left(\kappaeg+\frac{\kappaef}{e_3c_1(1-\beta_2)}+e_2\right)\Delta_{k(0)-1},
\end{align}
contradicting \eqref{eq_rho_temp5}. That is, \eqref{eq_rho_temp6} is false and so together with \eqref{eq_rho_temp5}, we have \eqref{eq_suc_safety}. Note that $p_{k(0)-1}$ is a fully linear model of $f$ in the unsuccessful step, then \lemref{lem_suc_safety} implies that iteration $k(0)-1$ is not an unsuccessful step, which leads to a contradiction. 
\end{proof}

\begin{remark}
    The proof of \lemref{lem_rho_min} fixes a small error in the proof of \cite[Lemma 3.8]{cartis2019derivative}: When deriving a bound for the step size $\norm{\bs_{k(0)-1}}$ at iteration $k(0)-1$ which is either a safety step or an unsuccessful step, the first inequality in (3.14) of \cite[Lemma 3.8]{cartis2019derivative} should be $\max(\gamma_S, \gammadec^{-1}) \rho_{k(0)-1}$ instead of $\min(\gamma_S, \gammadec^{-1}) \rho_{k(0)-1}$. This affects the bound of $\Delta_{k(0)-1}$ in (3.16) of \cite[Lemma 3.8]{cartis2019derivative}. Setting $\tau_{k(0)-1} = 1$ in our proof, we follow the same approach as \cite{cartis2019derivative} to conclude that the safety step is not called. To deal with the case of an unsuccessful step, we adjust the constant in the expression of $\rho_{\min}$ so that we can bound $\Delta_{k(0)-1}$ in \eqref{eq_rho_temp4.5} similar to (3.16) in \cite[Lemma 3.8]{cartis2019derivative}. After these corrections, the result in \cite[Lemma 3.8]{cartis2019derivative} remains unchanged except that $\rho_{\min}$ has a different expression
    \begin{equation*}
        \rho_{\min} \coloneqq \min\left(\Delta_{0}^{\textnormal{init}}, \frac{\omega_C \epsilon}{\kappaeg + 1/\mu}, \frac{c_2 \gammadec\alpha_1 \epsilon_g}{2\kappa_{H}}, \frac{c_2 \gammadec\alpha_1}{2}\left(\kappaeg + \frac{2\kappaef}{c_1(1-\eta_2)}\right)^{-1}\epsilon\right),
    \end{equation*}
    which does not affect any of the results proved later in \cite{cartis2019derivative}.
\end{remark}

Our convergence results below are based on \cite[Chapter 10]{conn2009introduction}. Moreover, we consider the possibility of safety phases and maintenance on $\rho_k$ as presented in \cite{cartis2019derivative}. However, we use a different criticality measure $\Psi_1$ and its inaccurate estimation $\overline{\eta}_1$ in our proofs. The following lemma proves the Lipschitz continuity of $\Psi_1$, which allows us to follow the convergence analysis in \cite[Chapter 10]{conn2009introduction}. Note that when $h$ is chosen to be an indicator function of a convex set, this lemma is equivalent to \cite[Theorem 3.4]{cartis2012adaptive}.

\begin{lemma} \label{lem_lip_psi}
Suppose that Assumptions~\ref{ass_smooth} and \ref{ass_regu} hold. Then $\Psi_1$ is Lipschitz continuous in $\operatorname{dom} h$. That is, for all $x, y \in \operatorname{dom} h$,
\begin{align}
    \abs{\Psi_1(x) - \Psi_1(y)} \leq L_{\Psi}\norm{x-y},
\end{align}
where constant $L_{\Psi} \coloneqq L_{\grad f} + 2L_h$.
\end{lemma}
\begin{proof}
Take any $x, y \in \operatorname{dom} h$. From \eqref{eq_cri_psi}, we have
\begin{align} \label{eq_psi_lip_temp1}
    \Psi_1(\bx) - \Psi_1(\by) &= h(\bx) - h(\by) + \min_{\norm{\bs} \leq 1} [\grad f(\by)^{T}\bs + h(\by + \bs)] \\
    &\quad - \min_{\norm{\bs} \leq 1} [\grad f(\bx)^{T}\bs + h(\bx + \bs)]. 
\end{align}
Since $\grad f$ and $h$ are continuous by Assumptions~\ref{ass_smooth} and \ref{ass_regu}, applying the Weierstrass Theorem, there exists $\bs_x, \bs_y \in B(\b0, 1) \cap \textnormal{dom}h$ such that 
\begin{align}
    \min_{\norm{\bs} \leq 1} \grad f(\bx)^{T}\bs + h(\bx + \bs) &= \grad f(\bx)^{T}\bs_x + h(\bx + \bs_x), \\
    \min_{\norm{\bs} \leq 1} \grad f(\by)^{T}\bs + h(\by + \bs) &= \grad f(\by)^{T}\bs_y + h(\by + \bs_y). \label{eq_psi_lip_temp2}
\end{align}
Plugging them into \eqref{eq_psi_lip_temp1}, we obtain
\begin{align*}
    \Psi_1(\bx) - \Psi_1(\by) &= h(\bx) - h(\by) + \left(\grad f(\by)^{T}\bs_y + h(\by + \bs_y)\right) \\
    &\quad -  \left(\grad f(\bx)^{T}\bs_x + h(\bx + \bs_x)\right) \\
    &= h(\bx) - h(\by) + \grad f(\by)^{T}\bs_y + h(\by+\bs_y) -  \grad f(\by)^{T}\bs_x - h(\by+\bs_x) \\
    &\quad + \left(\grad f(\by) - \grad f(\bx)\right)^{T}\bs_x + h(\by + \bs_x) - h(\bx + \bs_x) \\
    &\leq \abs{h(\bx) - h(\by)} + \abs{\left(\grad f(\by) - \grad f(\bx)\right)^{T}\bs_x} + \abs{h(\by + \bs_x) - h(\bx + \bs_x)} \\
    &\leq (L_{\grad f} + 2L_h)\norm{\bx-\by}.
\end{align*}
In the first inequality, we use the triangle inequality and the fact that $\bs_y$ is a minimizer of \eqref{eq_psi_lip_temp2}. The last line follows from  Assumptions~\ref{ass_smooth} and \ref{ass_regu}, and $\|\bs_x\|,\|\bs_y\|\leq 1$. 
Since we can interchange the role of $\bx$ and $\by$ in the above argument, $\Psi_1$ is Lipschitz continuous in $\operatorname{dom} h$ with Lipschitz constant $L_{\grad f} + 2L_h$.
\end{proof}

The following convergence results consider the case when the number of successful iterations is finite.

\begin{lemma} \label{lem_finite_delta}
Suppose that Assumptions~\ref{ass_smooth}, \ref{ass_regu} and \ref{ass_boundedhess} hold. If there are finitely many successful iterations, then $\lim_{k \rightarrow \infty}\Delta_k = 0$ and $\lim_{k \rightarrow \infty} \Psi_1(\bx_k) = 0$.
\end{lemma}
\begin{proof}
Let us consider iterations that come after the last successful iteration $k_0$. From \algref{alg_dfo} and \lemref{lem_delta_bound}, we know that we can have either one model-improving phase or finite many (uniformly bounded by an integer $N$) iterations of \algref{alg_geo} in the criticality phase before the model is modified to be fully linear.
Therefore, we have an infinite number of iterations where the model is fully linear. That is, there are infinitely many iterations that are either safety or unsuccessful steps. In either case, the trust region radius is reduced by a factor $\max\{\alpha_2, \omega_S, \gammadec\}<1$. Since $\Delta_k$ is increased only in successful iterations, then $\Delta_k$ converges to zero.
For each iteration $k > k_0$, let $i_{k}$ denote the index of the first iteration after $k_0$ for which $p_k$ is a fully linear model for $f$. Since $\Delta_k$ is reduced in \algref{alg_geo}, then as $k\to\infty$, 
\begin{align} \label{eq_finite_temp1}
    \norm{\bx_{k} - \bx_{i_k}} \leq N\Delta_{k} \to 0.
\end{align}
Note that we can rewrite $\Psi_1$ and apply the triangle inequality to obtain
\begin{align*}
    \Psi_1(\bx_k) &\leq \abs{\Psi_1(\bx_k) - \Psi_1(\bx_{i_k})} + \abs{\Psi_1(\bx_{i_k}) - \eta_1(\bx_{i_k})} + \abs{\eta_1(\bx_{i_k}) - \overline{\eta}_1(\bx_{i_k})} + \overline{\eta}_1(\bx_{i_k}) \\
    &\leq L_{\Psi}\norm{\bx_{k} - \bx_{i_k}} + \kappaeg\Delta_{i_k} + e_2 \Delta_{i_k} + \overline{\eta}_1(\bx_{i_k}),
    \stepcounter{equation}\tag{\theequation}\label{eq_finite_temp2}
\end{align*}
which follows from \lemref{lem_lip_psi}, \lemref{lem_crit_perturb} and the approximation error of $\eta_{1}$ given by \algref{alg_dfo}. We observe that $\overline{\eta}_1(\bx_{i_k}) \to 0$ as $k \to \infty$. Otherwise, \eqref{eq_suc_safety} holds for sufficiently large $k$ and then by \lemref{lem_suc_safety}, iteration $k$ must be a safety phase. However, when $k$ is large enough, \lemref{lem_step_bound} implies that $\norm{\bs_k} \geq c_2 \tau_k \Delta_k \geq \tau_k \gamma_S \Delta_k$,
\end{proof}

\begin{lemma} \label{lem_conv_delta}
Suppose that Assumptions~\ref{ass_smooth}, \ref{ass_regu} and \ref{ass_boundedhess} hold. Then we have $\lim_{k \rightarrow \infty}\Delta_k = 0$ and hence $\lim_{k \rightarrow \infty} \rho_k = 0$. 
\end{lemma}
\begin{proof}
This proof is similar to \cite[Lemma 10.9]{conn2009introduction}. Let $\mcS$ denote the set of successful iterations. When $\abs{\mcS}<\infty$, this follows from \lemref{lem_finite_delta}. Now we consider the case when $\mcS$ is infinite. For any $k\in\mcS$, we have
\begin{align} 
    \Phi(\bx_k) - \Phi(\bx_{k+1}) \geq \beta_1[m_k(\bx_k) - m_k(\bx_k + \bs_k)].
\end{align}
Applying \lemref{lem_mk_dec}, \lemref{lem_eta_bound} and \assref{ass_boundedhess}, we obtain that 
\begin{align} \label{eq_conv_delta_temp1}
    \Phi(\bx_k) - \Phi(\bx_{k+1}) \geq \beta_1 e_3 c_1 \min\left\{e_1 \epsilon_C, \Delta_k / \mu\right\}\min\left\{\Delta_k, \frac{\min\left\{e_1 \epsilon_C, \Delta_k / \mu\right\}}{\kappa_H}\right\}.
\end{align}
Since $\mcS$ is infinite and $\Phi$ is bounded below by \assref{ass_regu}, then the left hand side of \eqref{eq_conv_delta_temp1} converges to $0$. Thus, $\lim_{k\in\mcS} \Delta_k = 0$. Let $k \not\in \mcS$ be the index of an iteration after the first successful iteration, and let $s_k$ denote the index of the last successful iteration before $k$. Since the trust region radius can be increased only during a successful iteration, and only by a factor of at most $\gammaincbar$, then $\Delta_k \leq \gammaincbar \Delta_{s_k}$. Since $s_k \in \mcS$, then $\Delta_{s_k} \to\ 0$ as $k\to0$ and hence $\Delta_k \to\ 0$ for $k \not\in \mcS$. Therefore, $\lim_{k\to\infty} \Delta_k = 0$. It follows immediately from the fact $\rho_k \leq \Delta_k$ that $\lim_{k\to\infty} \rho_k = 0$.
\end{proof}

\begin{lemma} \label{lem_conv_liminf}
Suppose that Assumptions~\ref{ass_smooth}, \ref{ass_regu} and \ref{ass_boundedhess} hold. Then
\begin{align}
    \liminf_{k\to\infty} \Psi_1(\bx_k) = 0.
\end{align}
\end{lemma}
\begin{proof}
Assume for contradiction that there is a constant $\epsilon>0$ such that $\Psi_1(\boldsymbol{x}_k) \geq \epsilon$ for all $k$. In this case, by \lemref{lem_rho_min}, we have that $\Delta_k \geq \rho_k \geq \rho_{\min}>0$ for all $k$, contradicting \lemref{lem_conv_delta}.
\end{proof}

\begin{theorem} \label{thm_convergence}
Suppose that Assumptions~\ref{ass_smooth}, \ref{ass_regu} and \ref{ass_boundedhess} hold. Then
\begin{align}
    \lim_{k\to\infty} \Psi_1(\bx_k) = 0.
\end{align}
\end{theorem}
\begin{proof}
This proof is based on \cite[Theorem 10.13]{conn2009introduction} and \cite[Theorem 3.12]{cartis2019derivative}. Let $\mcS$ and $\mcM$ denote the set of successful and model-improving iterations, respectively. Our theorem holds when $\abs{\mcS}<\infty$ by \lemref{lem_finite_delta}. Assume, for the purpose of establishing a contradiction, that there exists a subsequence $\{t_i\} \in \mcS$ of successful iterations such that
\begin{align}
    \Psi_1(\bx_{t_i}) \geq 2\epsilon > 0,
\end{align}
for some $\epsilon>0$ and for all $i$. From \lemref{lem_conv_liminf}, we know that for each $t_i$, there exists a first successful iteration $l_i>t_i$ such that $\Psi_1(\bx_{l_i})<\epsilon$. Now we consider iterations whose indices are in the set $\mcK$ defined by
\begin{align}
    \mcK \coloneqq \bigcup_{i\in\N}\{k\in\N: t_i \leq k < l_i\},
\end{align}
where $t_i$ and $l_i$ belong to the two subsequences defined above. Therefore, for every $k \in \mcK$, we have $\Psi_1(\bx_{k}) \geq \epsilon$. By \lemref{lem_eta_bound}, $\overline{\eta}_1(\bx_k) \geq \epsilon_g > 0$ for $k \in \mcK$. From \lemref{lem_conv_delta}, $\lim_{k\to\infty} \Delta_k = 0$. Therefore, \lemref{lem_step_bound} implies that $\norm{\bs_k} \geq c_2 \tau_k \Delta_k \geq \tau_k \gamma_S \Delta_k$ if $k$ is sufficiently large.
\begin{align} \label{eq_conv_psi_temp1}
    \norm{\bx_{t_i} - \bx_{l_i}} \leq \sum_{\substack{k = t_i \\k \in \mcK}}^{l_i-1}\norm{\bx_{k} - \bx_{k+1}} = \sum_{\substack{k = t_i \\k \in \mcK \cap \mcS}}^{l_i-1} \norm{\bx_{k} - \bx_{k+1}} \leq \sum_{\substack{k = t_i \\k \in \mcK \cap \mcS}}^{l_i-1} \Delta_k.
\end{align}
For each $k \in \mcK \cap \mcS$, using $\overline{\eta}_1(\bx_k) \geq \epsilon_g$ and \assref{ass_boundedhess}, we obtain
\begin{align}
    \Phi(\bx_k) - \Phi(\bx_{k+1}) &\geq \beta_1[m_k(\bx_k) - m_k(\bx_k + \bs_k)] \\
    &\geq \beta_1 e_3 c_1 \overline{\eta}_1(\bx_k)\min\left\{\Delta_k, \frac{\overline{\eta}_1(\bx_k)}{\max\{1,\norm{H_k}\}}\right\} \\
    &\geq \beta_1 e_3 c_1 \epsilon_g \min\left\{\Delta_k, \frac{\epsilon_g}{\kappa_H}\right\}
\end{align}
Since $\lim_{k\to\infty} \Delta_k = 0$ by \lemref{lem_conv_delta}, then $\min\{\Delta_k, \epsilon_g/\kappa_H\} = \Delta_k$ for $k$ large enough. Thus,
\begin{align} \label{eq_conv_psi_temp2}
    \Delta_k \leq \frac{1}{\beta_1 e_3 c_1 \epsilon_g} [\Phi(\bx_k) - \Phi(\bx_{k+1})]
\end{align}
Combining \eqref{eq_conv_psi_temp1} and \eqref{eq_conv_psi_temp2}, we have
\begin{align*} 
    \norm{\bx_{t_i} - \bx_{l_i}} &\leq \sum_{\substack{k = t_i \\k \in \mcK \cap \mcS}}^{l_i-1} \frac{1}{\beta_1 e_3 c_1 \epsilon_g} [\Phi(\bx_k) - \Phi(\bx_{k+1})] \\
    & = \frac{1}{\beta_1 e_3 c_1 \epsilon_g}\sum_{\substack{k = t_i \\k \in \mcK}}^{l_i-1} [\Phi(\bx_k) - \Phi(\bx_{k+1})] \\
    &=  \frac{1}{\beta_1 e_3 c_1 \epsilon_g}[\Phi(\bx_{t_i}) - \Phi(\bx_{l_i})].
    \stepcounter{equation}\tag{\theequation}\label{eq_conv_psi_temp3}
\end{align*}
Since $\Phi$ is bounded below by \assref{ass_regu} and the sequence $\{\Phi(\bx_{k})\}$ is monotone decreasing, then the right hand side of \eqref{eq_conv_psi_temp3} converges to zero. Therefore, 
\begin{align}
    \lim_{i\to\infty} \norm{\bx_{t_i} - \bx_{l_i}} = 0.
\end{align}
However, from \lemref{lem_lip_psi}, we have that $\abs{\Psi_1(\bx_{t_i}) - \Psi_1(\bx_{l_i})} \leq L_{\Psi}\norm{\bx_{t_i} - \bx_{l_i}} \to 0$ as $i\to\infty$. This contradicts our construction of subsequences $\{t_i\}$ and $\{l_i\}$ that $\Psi_1(\bx_{t_i}) - \Psi_1(\bx_{l_i}) \geq \epsilon > 0$. By the principle of contradiction, we conclude that $\lim_{k\to\infty}\Psi_1(\bx_k) = 0$.
\end{proof}

\subsection{Worst-Case Complexity}

Now we study the worst-case complexity of \algref{alg_dfo} following the approach of \cite{cartis2019derivative}. We bound the number of iterations and objective evaluations until $\Psi_1(\bx_{k}) < \epsilon$, where the existence of bounds is guaranteed by \lemref{lem_conv_liminf} for each optimality level $\epsilon$. Let $i_\epsilon$ be the last iteration before $\Psi_1(\bx_{i_\epsilon+1}) < \epsilon$ for the first time. We classify all iterations until iteration $i_\epsilon$ (inclusive) into five types with descriptions and symbols below:
\begin{enumerate}
    \item Criticality Iteration ($\mcCi$): the set of criticality iterations $k \leq i_\epsilon$. Moreover, we collect all iterations in $\mcCi$ where $\Delta_k$ is not reduced into a set $\mcCMi$. That is, $\mcCMi$ is the set of the first iteration of every call of \algref{alg_geo}. We denote the remaining iterations in $\mcCi$ by $\mcCUi \coloneqq \mcCi - \mcCMi$.
    \item Safety Iteration ($\mcSi$): the set of iterations $k \leq i_{\epsilon}$ where the safety phase is called.
    \item Successful Iteration ($\mcS$): the set of successful iterations $k \leq i_{\epsilon}$.
    \item Model Improvement Iteration ($\mcMi$): the set of iterations $k \leq i_{\epsilon}$ where the model improvement phase is called.
    \item Unsuccessful Iteration ($\mcUi$): the set of unsuccessful iterations $k \leq i_{\epsilon}$.
\end{enumerate}
We first count the number of iterations up to iteration $i_\epsilon$ that are successful in the following lemma.

\begin{lemma} \label{lem_count_suc}
Suppose that Assumptions~\ref{ass_smooth}, \ref{ass_regu} and \ref{ass_boundedhess} hold. Then 
\begin{align} \label{eq_count_suc}
    \abs{\mcSi} \leq \frac{\Phi(\bx_0)-\flow-\hlow}{\beta_1 e_3 c_1}\max\left\{\kappa_H\epsilon_g^{-2},\epsilon_g^{-1}\rho_{\min}^{-1}\right\},
\end{align}
where $\epsilon_g$ is defined in \eqref{eq_eta_bound} and $\rho_{\min}$ in \eqref{eq_rho_min}.
\end{lemma}
\begin{proof}
For all $k \in \mcSi$, by \lemref{lem_mk_dec} and \assref{ass_boundedhess}, we have that 
\begin{align}
    \Phi(\bx_k)-\Phi(\bx_{k+1}) \geq \beta_1[m_k(\bzero)-m_k(\bs_k)] \geq \beta_1 e_3 c_1 \overline{\eta}_1(\bx_{k})\min\left\{\Delta_k, \frac{\overline{\eta}_1(\bx_{k})}{\kappa_H}\right\}.
\end{align}
From \lemref{lem_eta_bound} and \lemref{lem_rho_min}, $\overline{\eta}_1(\bx_k) \geq \epsilon_g$ and $\Delta_k \geq \rho_{k} \geq \rho_{\min}$ 
for any $k \geq i_\epsilon$. This means
\begin{align} \label{eq_count_suc_temp1}
    \Phi(\bx_k)-\Phi(\bx_{k+1}) \geq \beta_1 e_3 c_1 \epsilon_g \min\left\{\rho_{\min}, \frac{\epsilon_g}{\kappa_H}\right\}.
\end{align}
Summing \eqref{eq_count_suc_temp1} over all $k \in \mcSi$, and noting that $\flow+\hlow \leq \Phi(\bx_k) \leq \Phi(\bx_0)$, we obtain
\begin{align}
    \Phi(\bx_0) - \flow-\hlow \geq \abs{\mcSi}\beta_1 e_3 c_1 \epsilon_g \min\left\{\rho_{\min},\frac{\epsilon_g}{\kappa_H}\right\},
\end{align}
from which \eqref{eq_count_suc} follows.
\end{proof}

\begin{lemma} \label{lem_count_relation}
Suppose that \assref{ass_smooth}, \assref{ass_regu} and \assref{ass_boundedhess} hold. Then we have the bounds
\begin{align}
    \abs{\mcCUi}+\abs{\mcFi}+\abs{\mcUi} &\leq \abs{\mcSi}\cdot\frac{\log\gammaincbar}{\abs{\log\alpha_3}}+\frac{1}{\abs{\log\alpha_3}}\log\left(\frac{\Delta_0^{\textnormal{init}}}{\rho_{\min}}\right), \\
    \abs{\mcCMi} &\leq \abs{\mcFi}+\abs{\mcSi}+\abs{\mcUi}, \\
    \abs{\mcMi} &\leq \abs{\mcCMi}+\abs{\mcCUi}+\abs{\mcFi}+\abs{\mcSi}+\abs{\mcUi},
\end{align}
where $\alpha_3 \coloneqq \max\{\omega_C,\omega_S,\gammadec,\alpha_2\}<1$ and $\rho_{\min}$ is defined in \eqref{eq_rho_min}.
\end{lemma}
\begin{proof}
See the proof of \cite[Lemma 3.14]{cartis2019derivative}.
\end{proof}

Following \cite{cartis2019derivative}, we make an additional assumption below, which can be easily satisfied by appropriate choices of parameters in \algref{alg_dfo}.

\begin{assumption} \label{ass_para}
  The algorithm parameter $\epsilon_C$ satisfies $e_1\epsilon_C=c_3\epsilon$ for some constant $c_3>0$.
\end{assumption}

\begin{theorem} \label{thm_count}
Suppose Assumptions~\ref{ass_smooth}, \ref{ass_regu}, \ref{ass_boundedhess} and \ref{ass_para} hold. 
Then the number of iterations $i_\epsilon$ until $\Psi_1(\bx_{i_\epsilon+1})<\epsilon$ is at most
\begin{align} \label{eq_count}
    &\left\lfloor\frac{4(\Phi(\bx_0)-\flow-\hlow)}{\beta_1 e_3 c_1}\left(1+\frac{\log\gammaincbar}{|\log\alpha_3|}\right)\max\left\{\frac{\kappa_H}{c_4^2\epsilon^2}, \frac{1}{c_4 c_5 \epsilon^2},\frac{1}{c_4 \Delta_0^{\textnormal{init}} \epsilon}\right\} \right.\\
    &\left.\quad +\frac{4}{|\log\alpha_3|}\max\left\{0,\log\left(\Delta_0^{\textnormal{init}}c_5^{-1}\epsilon^{-1}\right)\right\}\right\rfloor,
\end{align}
where $c_4 \coloneqq \min\{e_1 c_3, 1/(1+(\kappaeg+e_2)\mu)\}$ and 
\begin{align}
    c_5 \coloneqq \min\left\{\frac{\omega_C}{\kappaeg+1/\mu+e_2},\frac{c_2 c_4\gammadec\alpha_1}{2\kappa_H},\frac{c_2 \gammadec\alpha_1}{2}\left(\kappaeg+\frac{\kappaef}{e_3c_1(1-\beta_2)}+e_2\right)^{-1}\right\}.
\end{align}
\end{theorem}
\begin{proof}
From \assref{ass_para} and the definition of $\epsilon_g$ in \eqref{eq_eta_bound}, we have $\epsilon_g = c_4 \epsilon$. Then $\rho_{\min}$ defined in \eqref{eq_rho_min} can be rewritten as $\rho_{\min} = \min\left\{\Delta_{0}^{\textnormal{init}}, c_5\epsilon\right\}$. Using \lemref{lem_count_relation}, the total number of iterations can be bounded by
\begin{align}
    i_{\epsilon} &= \abs{\mcCMi} +  \abs{\mcCUi} + \abs{\mcFi} + \abs{\mcSi} + \abs{\mcMi} + \abs{\mcUi} \\
    &\leq 2\left(\abs{\mcCMi} +  \abs{\mcCUi} + \abs{\mcFi} + \abs{\mcSi} + \abs{\mcUi}\right) \\
    &\leq 2\abs{\mcCUi} + 4\left(\abs{\mcFi} + \abs{\mcSi} + \abs{\mcUi}\right) \\
    &\leq 4\left(\abs{\mcCUi} + \abs{\mcFi} + \abs{\mcUi}\right) + 4\abs{\mcSi} \\
    &\leq 4\abs{\mcSi}\left(1+\frac{\log\gammaincbar}{\abs{\log\alpha_3}}\right)+\frac{4}{\abs{\log\alpha_3}}\log\left(\frac{\Delta_0^{\textnormal{init}}}{\rho_{\min}}\right).
\end{align}
Therefore, \eqref{eq_count} follows from this and \lemref{lem_count_suc}.
\end{proof}

\begin{corollary} \label{cor_complexity}
Suppose Assumptions~\ref{ass_smooth}, \ref{ass_regu}, \ref{ass_boundedhess} and \ref{ass_para} hold. For $\epsilon \in (0,1]$, the number of iterations $i_\epsilon$  until $\Psi_1(\boldsymbol{x}_{i_\epsilon+1})<\epsilon$ for the first time is at most $\bigO(\kappa_H\kappa_d^{2}\epsilon^{-2})$, where $\kappa_d:=\max(\kappaef,\kappaeg)$. 
\end{corollary}
\begin{proof}
We note that $c_4^{-1} = \bigO(\kappaeg)=\bigO(\kappa_d)$ from \thmref{thm_count}. 
\begin{align}
    c_5^{-1} = \bigO\left(\max\left\{\kappaeg,\kappa_H c_4^{-1},\kappaef+\kappaeg\right\}\right) = \bigO(\kappa_H\kappa_d).
\end{align}
To leading order in \eqref{eq_count}, the number of iterations is
\begin{align}
    \bigO\left(\max\{\kappa_H c_4^{-2}, c_4^{-1} c_5^{-1}\}\epsilon^{-2}\right) = \bigO\left(\kappa_H \kappa_d^2\epsilon^{-2}\right),
\end{align}
as required.
\end{proof}

\section{Adaptation to Regularized Nonlinear Least-Squares} \label{sec_nlls}

We now consider the specific case of regularized nonlinear least-squares problems, where \eqref{eq_prob} becomes
\begin{align}
	\min_{\bx\in\R^n} \Phi(\bx) := f(\bx) + h(\bx) = \frac{1}{2}\|\br(\bx)\|^2 + h(\bx), \label{eq_prob_nlls}
\end{align}
for some function $\br:\R^n\to\R^m$.
As above, we assume that $\br$ is continuously differentiable with Jacobian $[J(\bx)]_{i,j} = \frac{\partial r_i(\bx)}{\partial x_{j}}$ but these derivatives are not accessible.

\begin{assumption} \label{ass_smooth_nlls}
	$\br(\bx)$ is continuously differentiable and its Jacobian is Lipschitz continuous with Lipschitz constant $L_{J}$ in the convex hull of  $\mathcal{B} \coloneqq \cup_k B(\boldsymbol{x}_k,\Delta_{\max})$. 
	Furthermore, $\br(\bx)$ and $J(\bx)$ are uniformly bounded in $\mathcal{B}$, i.e.~$\norm{\br(\bx)} \leq \rmax$ and $\norm{J(\bx)} \leq \jmax$ for all $\bx \in \mathcal{B}$. 
\end{assumption}

Note that \assref{ass_smooth_nlls} (taken from \cite[Assumption 3.1]{cartis2019derivative}) automatically gives us the smoothness requirements from \secref{sec_convergence}.

\begin{lemma}[Lemma 3.2, \cite{cartis2019derivative}] \label{lem_lip_gradf}
If \assref{ass_smooth_nlls} holds, then \assref{ass_smooth} holds with
\begin{align} \label{eq_lip_gradf}
    L_{\grad f} \coloneqq \rmax L_{J}+\jmax^{2}.
\end{align}
and $\flow=0$.
\end{lemma}

We are now in a position to precisely specify how we construct the model $p_k$ \eqref{eq_pk} to be used in $m_k$ \eqref{eq_model}.
Our construction is based on a linear interpolating model for $\br$, a derivative-free analog of the Gauss-Newton method \cite{cartis2019derivative}.
At $k$-th iteration, the linear Taylor series of the residual $\br(\bx)$ around $\bx_{k}$ is given by
\begin{align}
    \br(\bx_{k}+\bs) \approx \br(\bx_{k}) + J(\bx_{k})\bs, 
\end{align}
Since $J(\bx)$ is not accessible, we maintain a set of $n+1$ interpolation points containing $\bx_{k}$, denoted by $Y_k \coloneqq\{\by_0, \cdots, \by_n\}$ with $\by_0 \coloneqq \bx_k$, and approximate $J(\bx_{k})$ by a matrix $J_k \in \R^{m \times n}$. 
That is, we define $\bem_{k}(\bx_k+\bs) \coloneqq \br(\bx_{k}) + J_k\bs$, where $J_k$ satisfies the following interpolation conditions
\begin{align} \label{eq_interpolation_nlls}
    \br(\by_t) = \bem_k(\by_t), \qquad \forall t = 1,\ldots, n.
\end{align}
The uniqueness of $J_k$ is guaranteed once the interpolation directions $\{\by_1 - \bx_k, \cdots, \by_n - \bx_k\}$ are linearly independent, and we say that $Y_k$ is \emph{poised} for linear interpolation \cite[Section 2.3]{conn2009introduction}. 
Our resulting quadratic model $p_k$ for $f$ is
\begin{align} \label{eq_pk_nlls}
    f(\bx_{k}+\bs) \approx p_k(\bx_{k}+\bs) &\coloneqq \frac{1}{2}\norm{\bem_{k}(\bx_{k}+\bs)}^2,
\end{align}
and so $\bg_k \coloneqq J_k^{T}\br(\bx_k)$ and $H_k \coloneqq J_k^{T}J_k$ in \eqref{eq_pk}. 
Therefore, our full model $m_k$ \eqref{eq_model} for $\Phi$ is
\begin{align} \label{eq_model_nlls}
    \Phi(\bx_k + \bs) \approx m_k(\bx_{k}+\bs) \coloneqq p_k(\bx_{k}+\bs) + h(\bx_k + \bs).
\end{align}

In this setting, we now have a specific version of \lemref{lem_kappa_generic} for our model construction \eqref{eq_pk_nlls}.

\begin{lemma}[Lemma 3.3, \cite{cartis2019derivative}] \label{lem_kappa}
Suppose \assref{ass_smooth_nlls} holds and $Y_k$ is $\Lambda$-poised in $B(\bx_k,\Delta_k)$. 
Then $p_k$ \eqref{eq_pk_nlls} is a fully linear model for $f$ in $B(\bx_k,\Delta_k)$ with
\begin{align}
    \kappaef &= \kappaeg+\frac{L_{\grad f}+(\kappaeg^{r}\Delta_{\max}+\jmax)^2}{2}
    \quad \text{and} \\
    \kappaeg &= L_{\nabla f}+\kappaeg^{r}\rmax+(\kappaeg^{r}\Delta_{\max}+\jmax)^2,
\end{align}
where $\kappaeg^{r}:=\frac{1}{2}L_{J}(\sqrt{n}C+2)$ and $C=\bigO(\sqrt{n} \Lambda)$.
We also have the bound
\begin{align}
    \norm{H_k} &\leq (\jmax^{k})^{2}=(\kappaeg^{r}\Delta_{\max}+\jmax)^{2}. \label{eq_nlls_hess_bdd}
\end{align}
\end{lemma}

We are now in a position to report the convergence and complexity results for \algref{alg_dfo} applied to \eqref{eq_prob_nlls}.

\begin{corollary} \label{cor_complexity_nlls}
	Suppose Assumptions~\ref{ass_smooth_nlls}, \ref{ass_regu}, \ref{ass_boundedhess} and \ref{ass_para} hold, and we run \algref{alg_dfo} with the model $m_k$ from \eqref{eq_model_nlls}.
	Then $\lim_{k\to\infty} \Psi_1(\bx_k) = 0$.
	For $\epsilon \in (0,1]$, the number of iterations $i_\epsilon$  until $\Psi_1(\boldsymbol{x}_{i_\epsilon+1})<\epsilon$ for the first time is at most $\bigO(\kappa_H\kappa_d^{2}\epsilon^{-2})$, where $\kappa_d = \bigO(n^2 L_J^2 \Lambda^2)$, and the number of evaluations of $\br$ until $i_\epsilon$ is at most $\bigO(\kappa_H\kappa_d^{2}n\epsilon^{-2})$.
\end{corollary}
\begin{proof}
	\lemref{lem_lip_gradf} shows \assref{ass_smooth} holds. Otherwise, we replace \lemref{lem_kappa_generic} with \lemref{lem_kappa} and apply \thmref{thm_convergence} and \corref{cor_complexity}.
	Our interpolation set contains $n+1$ points, so we require no more than $n+1$ objective evaluations for each iteration. Thus, the total number of objective evaluations up to iteration $i_\epsilon$ is $\bigO(n i_{\epsilon})$, as expected. 
	The result $\kappa_d = \max(\kappaef,\kappaeg) = \bigO(n^2 L_J^2 \Lambda^2)$ follows from \lemref{lem_kappa}.
\end{proof}

The bound \eqref{eq_nlls_hess_bdd} suggests that $\kappa_H = \bigO(\kappa_d)$ is a reasonable estimate (even though it only holds in iterations where $p_k$ is fully linear).
This would give a final complexity bound of $\bigO(n^6 \epsilon^{-2})$ iterations or $\bigO(n^7 \epsilon^{-2})$ evaluations of $\br$.

\begin{remark}
	Regarding the constant $C$ in \lemref{lem_kappa}, we point the reader to \cite[Remark 3.20]{cartis2019derivative} for a detailed discussion of the dependency of $C$ on $n$. In particular, the results from \cite{garmanjani2016trust} effectively assume $C=\bigO(\Lambda)$ and so report an improved dependency on $n$ in their final complexity bound.
	In the case of \corref{cor_complexity_nlls}, we would instead get $\kappa_d=\bigO(n)$ and a final complexity bound of $\bigO(n^3 \epsilon^{-2})$ iterations or $\bigO(n^4 \epsilon^{-2})$ evaluations of $\br$.
\end{remark}

\paragraph{Subproblem solutions}
Although \algref{alg_dfo} is for the general problem \eqref{eq_prob}, the trust-region subproblem \eqref{eq_trsub} (where we require a solution with accuracy \eqref{eq_mk_dec}) is now minimizing the sum of a nonconvex quadratic function and nonsmooth convex regularizer subject to a ball constraint.
However, this problem becomes convex in the nonlinear least-squares case, where $m_k$ \eqref{eq_model_nlls} is convex (since $H_k = J_k^T J_k$ is positive semidefinite).

We also have the convex subproblem of computing the criticality measure \eqref{eq_min_criticality_alg1} to accuracy \eqref{eq_algdfo_eta_error} (or \eqref{eq_min_criticality_alg2} to accuracy \eqref{eq_alggeo_eta_error} in the case of \algref{alg_geo}).
This is also convex, requiring the minimization of a linear function plus a nonsmooth convex regularizer subject to a ball constraint.

We will discuss how both (convex) subproblems can be solved to the desired accuracy in \secref{sec_sfista}.

\section{Smoothing-Based Algorithm} \label{sec_alg_smoothing}

We now introduce our second approach for solving \eqref{eq_prob}, based on the smoothing technique from \cite{garmanjani2016trust}.
Here, we introduce a smoothing parameter $\gamma>0$ and build an approximation $\Phi^{\gamma}(\bx) \approx \Phi(\bx)$ with $\Phi^{\gamma}$ smooth and $\Phi^{\gamma}\to \Phi$ as $\gamma\to 0$.
For decreasing values of $\gamma$, we approximately minimize $\Phi^{\gamma}$ using a DFO method suitable for smooth objectives.
This algorithm is summarized in Algorithm~\ref{alg_dfo_smoothing}.

\begin{algorithm}[tbh]
\scriptsize
\begin{algorithmic}[1]
\Require Starting point $\bx_0\in\operatorname{dom} h$, initial smoothing parameter $\gamma_0>0$.
\vspace{0.2em}
\Statex \underline{Parameters:} trust-region termination function $d:(0,\infty)\to(0,\infty)$, smoothing update parameter $\sigma\in(0,1)$.
\For{$j=0,1,2,\ldots$}
    \State Approximately minimize $\Phi^{\gamma_j}(\bx)$ starting from $\bx_j$ using a globally convergent DFO trust-region method, terminating when it reaches minimnum trust-region radius $\Delta_{j,k} < d(\gamma_j)$ and returning approximate minimizer $\bx_{j+1}$. \label{ln_subproblem}
    \State Set $\gamma_{j+1}=\sigma\gamma_j$.
\EndFor
\end{algorithmic}
\caption{Smoothing model-based DFO method for \eqref{eq_prob} \cite[Algorithm 4.1]{garmanjani2016trust}}
\label{alg_dfo_smoothing}
\end{algorithm}

For Algorithm~\ref{alg_dfo_smoothing} to converge, we require the following properties about the smoothed function $\Phi^{\gamma}$.

\begin{assumption} \label{ass_smoothing}
    For any $\gamma>0$, the function $\Phi^{\gamma}:\R^n\to\R$ satisfies:
    \begin{enumerate}[label=(\alph*)]
        \item $\Phi^{\gamma}$ has $L_{\Phi}(\gamma)$-Lipschitz continuous gradient for some $L_{\Phi}(\gamma)>0$ and is bounded below on $\R^n$.
        \item For any $\bx\in\R^n$, $\lim_{\gamma\to 0^{+}} \Phi^{\gamma}(\bx)=\Phi(\bx)$.
    \end{enumerate}
\end{assumption}

In \cite{garmanjani2016trust}, the sufficient decrease condition in the inner loop (line \ref{ln_subproblem} of Algorithm~\ref{alg_dfo_smoothing}) is required to be based on the ratio
\begin{align}
    R_k \coloneqq \frac{\Phi^{\gamma_j}(\bx_k) - \Phi^{\gamma_j}(\bx_k + \bs_k) - c\Delta_k^{p}}{m_k(\bx_k) - m_k(\bx_k + \bs_k)} \label{eq_ratio_modified}
\end{align}
for some $c>0$ and $p>1$, which is a slight modification of \eqref{eq_ratio}.

The global convergence for Algorithm~\ref{alg_dfo_smoothing} is given by the following result.

\begin{theorem}[Theorem 4.3, \cite{garmanjani2016trust}] \label{thm_smoothing_convergence}
    Suppose Assumption~\ref{ass_smoothing} holds, and all model Hessians used in line~\ref{ln_subproblem} of Algorithm~\ref{alg_dfo_smoothing} are uniformly bounded.
    Then if $\lim_{\gamma\to 0^{+}} d(\gamma)=0$ and $\lim_{\gamma\to 0^{+}} L_{\Phi}(\gamma) d(\gamma) = 0$, then $\lim_{j\to\infty} \|\grad \Phi^{\gamma_j}(\bx_j)\|=0$.
\end{theorem}

We also note that \cite[Corollary 4.1]{garmanjani2016trust} provides a worst-case complexity bound for Algorithm~\ref{alg_dfo_smoothing} of $\bigO(|\log(\epsilon)| \epsilon^{-3})$ function evaluations to reach an outer iteration $j$ with $\|\grad \Phi^{\gamma_j}(\bx_j)\| < \bigO(\epsilon)$ (with constants depending on the dimension $n$), provided we choose $p=3/2$ in \eqref{eq_ratio_modified} and trust-region termination function $d(\gamma)=\gamma^2$, and $L_{\Phi}(\gamma)=\bigO(1/\gamma)$ in Assumption~\ref{ass_smoothing}.

\subsection{Adaptation to Regularized Least-Squares}
We now outline how we adapt Algorithm~\ref{alg_dfo_smoothing} to the regularized least-squares setting \eqref{eq_prob_nlls}.
Our approach uses DFO-GN \cite[Algorithm 1]{cartis2019derivative}---essentially Algorithm~\ref{alg_dfo} with $h=0$---with modified interpolation models similar to \eqref{eq_model_nlls}.

For the objective $\Phi$ \eqref{eq_prob_nlls}, our smoothed approximation will be 
\begin{align}
    \Phi^{\gamma}(\bx) \coloneqq f(\bx) + M_{h}^{\mu(\gamma)}(\bx) = \frac{1}{2}\|\br(\bx)\|^2 + M_{h}^{\mu(\gamma)}(\bx), \label{eq_smoothed_nlls}
\end{align}
where 
\begin{align}
    M_h^{\mu} (\bx) \coloneqq \min_{\bz \in \operatorname{dom} h}\left\{h(\bz)+\frac{1}{2\mu}\norm{\bz-\bx}^{2}\right\},
\end{align}
is the Moreau envelope of $h$ for the parameter $\mu>0$.
Note that for a given $\gamma$, we will choose a value $\mu(\gamma)$ for the Moreau envelope parameter.
We will use the following standard properties of $M_{h}^{\mu}$.

\begin{lemma} \label{lemma_smooth}
    Suppose $h$ satisfies \assref{ass_regu}. Then for any $\mu>0$, we have:
    \begin{enumerate}[label=(\alph*)]
        \item $M_{h}^{\mu}(\bx) \leq h(\bx) \leq M_{h}^{\mu}(\bx) + \frac{L_h^2\mu}{2}$ for all $\bx \in \operatorname{dom} h$.
        \item $M_h^{\mu}$ is a convex and differentiable function on $\operatorname{dom} h$, and $\grad M_h^{\mu}$ is Lipschitz continuous with Lipschitz constant $\frac{1}{\mu}$.
    \end{enumerate}
\end{lemma}
\begin{proof}
	See for example \cite[Theorems 6.55, 6.60 \& 10.51]{beck2017first}.
\end{proof}

From this, we see that our smoothed approximation \eqref{eq_smoothed_nlls} is a suitable choice for Algorithm~\ref{alg_dfo_smoothing}.

\begin{lemma} \label{lem_smoothing_assumptions}
    Suppose $\br$ satisfies \assref{ass_smooth_nlls} and $h$ satisfies \assref{ass_regu}.
    If $\mu(\gamma) = \Theta(\gamma)$ as $\gamma\to 0^{+}$ then $\Phi^{\gamma}$ \eqref{eq_smoothed_nlls} satisfies Assumption~\ref{ass_smoothing} with $L_{\Phi}(\gamma)=\bigO(1/\gamma)$.
\end{lemma}
\begin{proof}
    This follows immediately from Assumptions~\ref{ass_smooth_nlls} and~\ref{ass_regu}, together with Lemmas~\ref{lem_lip_gradf} and \ref{lemma_smooth}.
\end{proof}

In DFO-GN (called in line~\ref{ln_subproblem} of Algorithm~\ref{alg_dfo_smoothing}), as in Section~\ref{sec_nlls}, we build linear interpolation models of $\br(\bx)$ of the form
\begin{align}
    \br(\bx_k+\bs) \approx \bem_k(\bx_k+\bs) \coloneqq \br(\bx_k) + J_k \bs,
\end{align}
for some $J_k$ derived from the interpolation conditions \eqref{eq_interpolation_nlls}, yielding our model
\begin{align}
    \Phi^{\gamma}(\bx_k+\bs) \approx m_k(\bx_k+\bs) \coloneqq \frac{1}{2}\|\bem_k(\bx_k+\bs)\|^2 + M_{h}^{\mu(\gamma)}(\bx_k+\bs). \label{eq_model_nlls_smoothing}
\end{align}
From Lemma~\ref{lem_kappa}, this model is a fully linear approximation for $\Phi^{\gamma}$ whenever the interpolation set used for \eqref{eq_interpolation_nlls} is $\Lambda$-poised, and so all geometry-improving procedures from DFO-GN can be used here.

It only remains to describe how the trust-region subproblem 
\begin{align}
    \min_{\|\bs\|\leq\Delta_k} m_k(\bx_k+\bs) = \frac{1}{2}\|\bem_k(\bx_k+\bs)\|^2 + M_{h}^{\mu(\gamma)}(\bx_k+\bs) \label{eq_smoothing_trs}
\end{align}
is solved, and the choice of $\mu(\gamma)$.
These are discussed in \secref{sec_sfista}. 

\subsection{Global Convergence}
\thmref{thm_smoothing_convergence} establishes the global convergence of Algorithm \ref{alg_dfo_smoothing}, but it does not provide any standard first-order optimality results. 
We now show that any accumulation point of the $\bx_j$ generated by Algorithm~\ref{alg_dfo_smoothing} is a Clarke-stationary point of $\Phi$.

\begin{theorem}
    Suppose the assumptions of \thmref{thm_smoothing_convergence} and \lemref{lem_smoothing_assumptions} hold.
    If $\bx^*$ is any accumulation point of the iterates $\bx_j$ of Algorithm~\ref{alg_dfo_smoothing}, then $\bm{0} \in \partial_{C} \Phi(\bx^*)$.
\end{theorem}
\begin{proof}
    From \thmref{thm_smoothing_convergence}, we have $\grad f(\bx_j) + \grad M_{h}^{\mu(\gamma_j)}(\bx_j) \to \bm{0}$, and so since $\bx_{j_k}\to\bx^*$ for some subsequence of iterates and $f$ is continuously differentiable, we have $\grad M_{h}^{\mu(\gamma_{j_k})}(\bx_{j_k}) \to -\grad f(\bx^*)$.
    From \cite[Theorem 5.5]{Jourani1999} and $\mu(\gamma_{j_k})\to 0$, we conclude that $-\grad f(\bx^*) \in \partial h(\bx^*)$.
    Finally, we have $\partial_{C}\Phi(\bx) = \grad f(\bx) + \partial h(\bx)$ from \cite[Corollary 1, p.~39]{clarke1990optimization} and the result follows.
\end{proof}

\section{Calculating Subproblem Solutions} \label{sec_sfista}

There are three places where we need an algorithm to calculate subproblem solutions: the trust-region subproblem \eqref{eq_trsub} in \algref{alg_dfo}, calculating $\overline{\eta}_1(\bx_k)$ in Algorithms~\ref{alg_dfo} and \ref{alg_geo} (in \eqref{eq_min_criticality_alg1} and \eqref{eq_min_criticality_alg2} respectively), and the trust-region subproblem \eqref{eq_smoothing_trs} in the smoothed DFO method Algorithm~\ref{alg_dfo_smoothing}.
In the case where $H_k$ is positive semidefinite, as in the case solving regularized nonlinear least-squares problems \eqref{eq_prob_nlls} (with models \eqref{eq_model_nlls} and \eqref{eq_model_nlls_smoothing}), this becomes practical.

In all three cases, our subproblems can be written in the form
\begin{align} \label{eq_sfista_min}
    \min_{\bd} G(\bd) := \bg^{T}\bd + \frac{1}{2}\bd^{T}H\bd + h(\bx+\bd) + I_C(\bd),
\end{align}
where $H$ is positive semidefinite and $I_C$ is an indicator function for the Euclidean ball $C := B(\b0,r)$.
This problem is the sum of a smooth convex function and two nonsmooth convex functions.
To solve this to arbitrary accuracy, we will use S-FISTA \cite[Chapter 10.8]{beck2017first}, a smoothed version of FISTA.

Using \lemref{lemma_smooth}, we approximate $h$ by $M_h^{\mu}$ to obtain a smoothed formulation of \eqref{eq_sfista_min}:
\begin{align}
    \min_{\bd} G_\mu(\bd) := f(\bx) + \bg^{T}\bd + \frac{1}{2}\bd^{T}H\bd + M_h^{\mu}(\bx+\bd) + I_C(\bd),
\end{align}
for some smoothing parameter $\mu>0$. 
Defining $F_\mu(\bd) \coloneqq f(\bx) + \bg^{T}\bd + \frac{1}{2}\bd^{T}H\bd + M_h^{\mu}(\bx+\bd)$, from \cite[Theorem 6.60]{beck2017first} we have
\begin{align}
    \grad F_\mu(\bd) = \bg + H\bd + \frac{1}{\mu}(\bx+\bd-\prox_{\mu h}(\bx+\bd)),
\end{align}
where the proximal operator $\prox_{\mu h}(\cdot)$ is defined by
\begin{align}
    \prox_{\mu h}(\by) \coloneqq \arg\min_{\bz \in \operatorname{dom}h}\left\{h(\bz)+\frac{1}{2\mu}\norm{\bz-\by}^2\right\}.
\end{align}
Note that the proximal operator of $I_C$ is the (well-defined) Euclidean projection $P_C$ onto the set $C$. 
The S-FISTA algorithm \cite[Chapter 10.8.4]{beck2017first}, specialized to the problem \eqref{eq_sfista_min} is given in \algref{alg_sfista}.

\begin{algorithm}
\begin{algorithmic}[1]
\Require smoothing parameter $\mu>0$.
\State Set $\bd^{0} = \by^{0} = \b0$, $t_0 = 1$, and step size $L=\norm{H}+\frac{1}{\mu}$. 
\For{k=0, 1, 2, \dots}
\State set $\bd^{k+1} = P_C\left(\by^{k}-\frac{1}{L}\nabla F_\mu(\by^{k})\right)$;
\State set $t_{k+1}=\frac{1+\sqrt{1+4t_k^2}}{2}$;
\State compute $\by^{k+1}=\bd^{k+1}+\left(\frac{t_k-1}{t_{k+1}}\right)(\bd^{k+1}-\bd^{k})$.
\EndFor
\end{algorithmic}
\caption{S-FISTA \cite[Chapter 10.8.4]{beck2017first} for solving \eqref{eq_sfista_min}}
\label{alg_sfista}
\end{algorithm}

The following theorem shows that for a given accuracy level $\epsilon$, the smoothing parameter $\mu$ and the number of iterations $K$ in \algref{alg_sfista} can be chosen carefully to guarantee optimality $\epsilon$ in $\bigO(\frac{1}{\epsilon})$ iterations.

\begin{theorem} \label{thm_sfista_accuracy}
Suppose \assref{ass_smooth}, \assref{ass_regu} and \assref{ass_boundedhess} hold. Consider the set $C \coloneqq \left\{\bd: \norm{\bd} \leq r\right\}$.
Let $\{\bd^{k}\}_{k \geq 0}$ be the sequence generated by S-FISTA. For an accuracy level $\epsilon>0$, if the smoothing parameter $\mu$ and the number of iterations $K$ are set as
\begin{align} \label{eq_sfista_para}
    \mu = \frac{2\epsilon}{L_h(L_h+\sqrt{L_h^2+2\norm{H}\epsilon})} \quad \textnormal{and} \quad K = \frac{r(2L_h+\sqrt{2\norm{H}\epsilon})}{\epsilon},
\end{align}
then for any $k \geq K$, it holds that $G(\bd^{k})-G(\boldsymbol{d}^*)\leq\epsilon$, where $\boldsymbol{d}^*$ is in the optimal set of problem in \eqref{eq_sfista_min}.
\end{theorem}
\begin{proof}
Substitute $(\alpha,\beta,\Gamma,L_f)=(1, \frac{L_h^2}{2}, r^2, \|H\|)$ into \cite[Theorem 10.57]{beck2017first}.
\end{proof}

\subsection{Trust-region subproblem}
We are now able to show that \algref{alg_sfista} can be used to solve \eqref{eq_trsub} for \algref{alg_dfo} in the case of nonlinear least-squares problems (where $m_k$ is given by \eqref{eq_model_nlls}).
Within \algref{alg_sfista}, we set $r \coloneqq \Delta_k$, $\bg \coloneqq \bg_k$, $H \coloneqq H_k$ and $\epsilon \coloneqq (1-e_3)c_1\overline{\eta}_1(\bx_k)\min\left\{\Delta_k,\frac{\overline{\eta}_1(\bx_k)}{\max\{1,\norm{H_k}\}}\right\}$.
We use $K$ iterations of \algref{alg_sfista} with smoothing parameter $\mu$ as per \thmref{thm_sfista_accuracy}.
\lemref{lem_mk_dec} below shows that this gives us the required decrease \eqref{eq_mk_dec}.

\begin{lemma} \label{lem_mk_dec}
Suppose \assref{ass_regu} holds and the step $\bs_k$ is calculated such that
\begin{align} \label{eq_alg_suff_dec}
	m_k(\bx_{k}+\bs_k)-m_k(\bx_{k}+\bs_k^*) \leq (1-e_3)c_1\overline{\eta}_1(\bx_k)\min\left\{\Delta_k,\frac{\overline{\eta}_1(\bx_k)}{\max\{1,\norm{H_k}\}}\right\},
\end{align}
where $\bs_k^{*}$ is an exact minimizer of \eqref{eq_trsub}.
Then the step $\bs_k$ satisfies \eqref{eq_mk_dec} with $c_1 \coloneqq \dfrac{1}{2}\min\left\{1, \dfrac{1}{\Delta_{\max}^{2}}\right\}$. 
\end{lemma}
\begin{proof}
This proof follows \cite[Lemma 11]{grapiglia2015convergence}, with the main difference being the use of $\overline{\eta}_1(\bx_k)$ instead of $\Psi_1(\bx_k)$. 

We begin by defining $\bs_k^*$ as a global minimizer of the subproblem \eqref{eq_trsub}.
Let us also define $\tilde{\bs}_k$ to be a global minimizer of $\tilde{l}(\bx_k,\bs)$ for $\bs\in B(\b0,\Delta_k)$ (which exists since $h$ is continuous by \assref{ass_regu}), and so $\eta_{\Delta_k}(\bx_k) = \tilde{l}(\bx_k,\b0) - \tilde{l}(\bx_k,\tilde{\bs}_k)$.
Note also that since $h$ is convex, we have, for any $\theta\in[0,1]$,
\begin{align}
	h(\bx+\theta\tilde{\bs}_k) \leq (1-\theta) h(\bx) + \theta h(\bx+\tilde{\bs}_k).
\end{align}
Hence, for any $\theta\in[0,1]$, we have
\begin{align}
&	m_k(\bx_k)-m_k(\bx_{k}+\bs_k^{*}) \nonumber \\
 &\qquad \geq m_k(\bx_k)-m_k(\bx_{k}+\theta\tilde{\bs}_k) \\
	&\qquad = h(\bx_k) -\theta \grad p_k(\bx_k)^T \tilde{\bs}_k - \frac{1}{2}\theta^2 \tilde{\bs}_k^T H_k \tilde{\bs}_k - h(\bx_k+\theta\tilde{\bs}_k) \\
	&\qquad \geq \theta h(\bx_k) -\theta \grad p_k(\bx_k)^T \tilde{\bs}_k - \frac{1}{2}\max\{1,\|H_k\|\}\Delta_k^2 \theta^2  - \theta h(\bx+\tilde{\bs}_k) \\
	&\qquad = \theta \eta_{\Delta_k}(\bx_k) - \frac{1}{2}\max\{1,\|H_k\|\}\Delta_k^2 \theta^2 . \label{eq_mk_dec_temp3}
\end{align}
Since $\theta\in[0,1]$ was arbitrary, \eqref{eq_mk_dec_temp3} is equivalent to
\begin{align}
	m_k(\bx_k)-m_k(\bx_{k}+\bs_k^{*}) &\geq \max_{0\leq\theta\leq1}\left\{\theta\eta_{\Delta_k}(\bx_k)-\frac{1}{2}\max\left\{1,\norm{H_k}\right\}\Delta_k^{2}\theta^{2}\right\}.
\end{align}
This is a concave quadratic in $\theta$, with constrained maximizer 
\begin{align}
 \theta^*=\min\left(1, \frac{\eta_1(\bx_k)}{\max\{1, \|H_k\|\}\Delta_k^2}\right).   
\end{align}
Using this value of $\theta^*$, we get
\begin{align}
    m_k(\bx_k)-m_k(\bx_{k}+\bs_k^{*}) &\geq \frac{1}{2}\min\left\{\eta_{\Delta_k}(\bx_k),\frac{[\eta_{\Delta_k}(\bx_k)]^{2}}{\max\left\{1,\norm{H_k}\right\}\Delta_k^{2}}\right\}.
\end{align}
Then \cite[Lemma 2.1]{cartis2011evaluation}\footnote{Stating that $\eta_r(\bx) \geq \min\{1,r\}\eta_1(\bx)$ for any $r>0$ and $\bx\in\R^n$.} gives
\begin{align} \label{eq_mk_dec_temp4}
    m_k(\bx_k)-m_k(\bx_{k}+\bs_k^{*}) &\geq \frac{1}{2}\min\{1,\Delta_k\}\eta_1(\bx_k)\min\left\{1,\frac{\min\{1,\Delta_k\}\eta_1(\bx_k)}{\max\left\{1,\norm{H_k}\right\}\Delta_k^{2}}\right\}.
\end{align}
If $\Delta_k \leq 1$, then \eqref{eq_mk_dec_temp4} reduces to 
\begin{align} \label{eq_mk_dec_temp5}
    m_k(\bx_k)-m_k(\bx_{k}+\bs_k^{*}) &\geq \frac{1}{2}\eta_1(\bx_k)\min\left\{\Delta_k,\frac{\eta_1(\bx_k)}{\max\left\{1,\norm{H_k}\right\}}\right\}.
\end{align}
If $\Delta_k \geq 1$, and noting that $\Delta_{\max}>1$ (by assumption in \algref{alg_dfo}), it follows from \eqref{eq_mk_dec_temp4} and $\Delta_k \leq \Delta_{\max} \leq \Delta_{\max}^{2}$ that 
\begin{align*}
    m_k(\bx_k)-m_k(\bx_{k}+\bs_k^{*}) &\geq \frac{1}{2}\eta_1(\bx_k)\min\left\{1,\frac{\eta_1(\bx_k)}{\max\left\{1,\norm{H_k}\right\}\Delta_k^{2}}\right\} \\
    &\geq \frac{1}{2\Delta_{\max}^2}\eta_1(\bx_k)\min\left\{\Delta_{\max}^2,\frac{\eta_1(\bx_k)}{\max\left\{1,\norm{H_k}\right\}}\right\} \\
    &\geq \frac{1}{2\Delta_{\max}^2}\eta_1(\bx_k)\min\left\{\Delta_{k},\frac{\eta_1(\bx_k)}{\max\left\{1,\norm{H_k}\right\}}\right\} \stepcounter{equation}\tag{\theequation}\label{eq_mk_dec_temp6}
\end{align*}
Therefore, combining \eqref{eq_mk_dec_temp5} and \eqref{eq_mk_dec_temp6}, since $\overline{\eta}_1(\bx_k) \leq \eta_1(\bx_k)$, we obtain:
\begin{align} \label{eq_mk_dec_temp7}
    m_k(\bx_k)-m_k(\bx_{k}+\bs_k^{*}) &\geq c_1\overline{\eta}_1(\bx_k)\min\left\{\Delta_k,\frac{\overline{\eta}_1(\bx_k)}{\max\left\{1,\norm{H_k}\right\}}\right\}.
\end{align}
Subtracting \eqref{eq_alg_suff_dec} from \eqref{eq_mk_dec_temp7} gives \eqref{eq_mk_dec} as required.
\end{proof}

\subsection{Computing criticality measures}
When we implement S-FISTA for the approximate criticality measure $\overline{\eta}_1$ in \algref{alg_dfo} (i.e.~in \eqref{eq_min_criticality_alg1} and \eqref{eq_min_criticality_alg2}), we choose $r \coloneqq 1$, $\bg\coloneqq \grad p_k(\bx_k)=\bg_k$ and $H \coloneqq 0$ 
in \eqref{eq_sfista_para}. 
The accuracy level is set to be $\epsilon \coloneqq \min\{(1-e_1)\epsilon_C, e_2\Delta_k^{\textnormal{init}}\}$ in \eqref{eq_min_criticality_alg1} and $\epsilon \coloneqq e_2 \omega_C^{i-1}\Delta_k^{\textnormal{init}}$ in \eqref{eq_min_criticality_alg2}, matching our requirements \eqref{eq_algdfo_eta_error} and \eqref{eq_alggeo_eta_error} respectively. 

\subsection{Smoothed trust-region subproblem}

In the case of solving the trust-region problem \eqref{eq_smoothing_trs} for Algorithm~\ref{alg_dfo_smoothing}, we follow \eqref{eq_sfista_para} from Theorem~\ref{thm_sfista_accuracy} and at each iteration use the value
\begin{align}
    \mu(\gamma) \coloneqq \frac{2\gamma}{L_h(L_h+\sqrt{L_h^2+2\norm{H_k}\gamma})}, \label{eq_mu_gamma}
\end{align}
where $\norm{H_k}=\|J_k^T J_k\|$ from \eqref{eq_model_nlls_smoothing}.
The total number of iterations $K$ is also given by substituting $\epsilon \leftarrow \gamma$ in \eqref{eq_sfista_para}.
This ensures that our trust-region subproblem is solved to global optimality level $\gamma$, which decreases to zero as Algorithm~\ref{alg_dfo_smoothing} progresses.

\section{Numerical Experiments} \label{sec_numerics}

\subsection{Implementation} \label{sec_implementation}
We investigate the performance of both approaches for solving the regularized least-squares problem \eqref{eq_prob_nlls} by implementing two modifications of the original DFO-LS software \cite{cartis2019improving}:
\begin{itemize}
    \item DFO-LSR (DFO-LS with Regularization): Algorithm~\ref{alg_dfo} with directly solves \eqref{eq_prob_nlls}.\footnote{Implementation available from \url{https://github.com/yanjunliu-regina/dfols}.}
    \item DFO-LSSR (DFO-LS with Smoothed Regularizion): Algorithm~\ref{alg_dfo_smoothing}, which iteratively smoothing \eqref{eq_prob_nlls} to \eqref{eq_smoothed_nlls}.\footnote{Implementation available from \url{https://github.com/khflam/dfols/}.} We used parameters $\sigma=0.1$, $\gamma_0=0.01$ and trust-region termination function $d(\gamma_j) = \mu(\gamma_j)^2$ for $\mu(\gamma)$ given by \eqref{eq_mu_gamma}.
\end{itemize}
All other algorithm parameters are set to their default values in DFO-LS.
As described in \secref{sec_sfista}, the different subproblems are solved using S-FISTA (Algorithm~\ref{alg_sfista}) with parameters as described above. 
In practice, we additionally terminate S-FISTA after $500$ iterations, which produces comparable numerical results to using the expected number of iterations, but improves the overall algorithm runtime. 

\subsection{Testing Setup} \label{sec_test_info}
The computational performance of DFO-LSR is compared with PyNOMAD, which is a Python interface for NOMAD \cite{le2011algorithm}, a solver for black-box optimization based on the Mesh Adaptive Direct Search (MADS) algorithm and is capable of solving nonsmooth problems.
We note that NOMAD is not aware of the problem structure in \eqref{eq_prob_nlls}, it simply can evaluate $\Phi(\bx)$ as a black-box, whereas our algorithms receive more problem information, namely $\br(\bx)$ and $h$ (plus its proximal operator).
We build our test suite based on \cite{more2009benchmarking}, a collection of $53$ unconstrained nonlinear least squares problems with dimension $2 \leq n \leq 12$ and $2 \leq m \leq 65$. For each problem, we set $f(\bx)$ to be the nonlinear least squares from \cite{more2009benchmarking} and $h(\bx)=\norm{\bx}_1$. Both solvers are tested additionally on our test suite where stochastic noise is introduced in the evaluations of the residuals $r_i$. Specifically, we introduce i.i.d. $\epsilon \sim N(0,\sigma^2)$ for each $i$ and $\bx$ and implement two unbiased noise models as follows:
\begin{itemize}
    \item Multiplicative Gaussian noise: we evaluate residual $\tilde{r}_i(\bx) = r_i(\bx)(1+\epsilon)$; and
    \item Additive Gaussian noise: we evaluate residual $\tilde{r}_i(\bx) = r_i(\bx) + \epsilon$.
\end{itemize}
For our testing, we took the noise level to be $\sigma=10^{-2}$.
Each solver was run 10 times on every problem with noise, and a maximum budget of $100(n+1)$ evaluations was given (for an $n$-dimensional problem).

\begin{figure}[H]
  \centering
  \begin{subfigure}[b]{0.48\textwidth}
    \includegraphics[width=\textwidth]{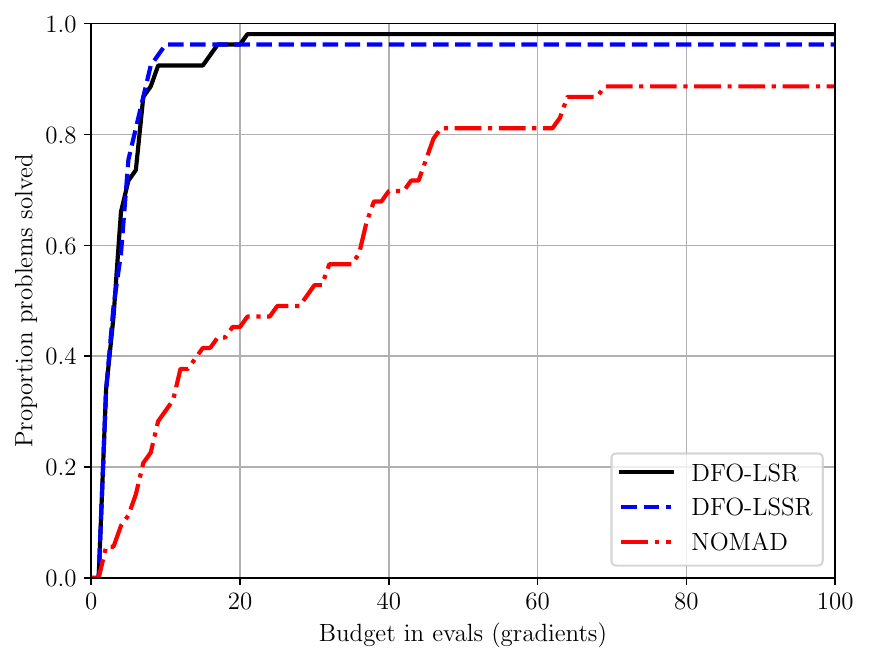}
    \caption{Data profile, $\tau=10^{-3}$}
  \end{subfigure}
  ~
  \begin{subfigure}[b]{0.48\textwidth}
    \includegraphics[width=\textwidth]{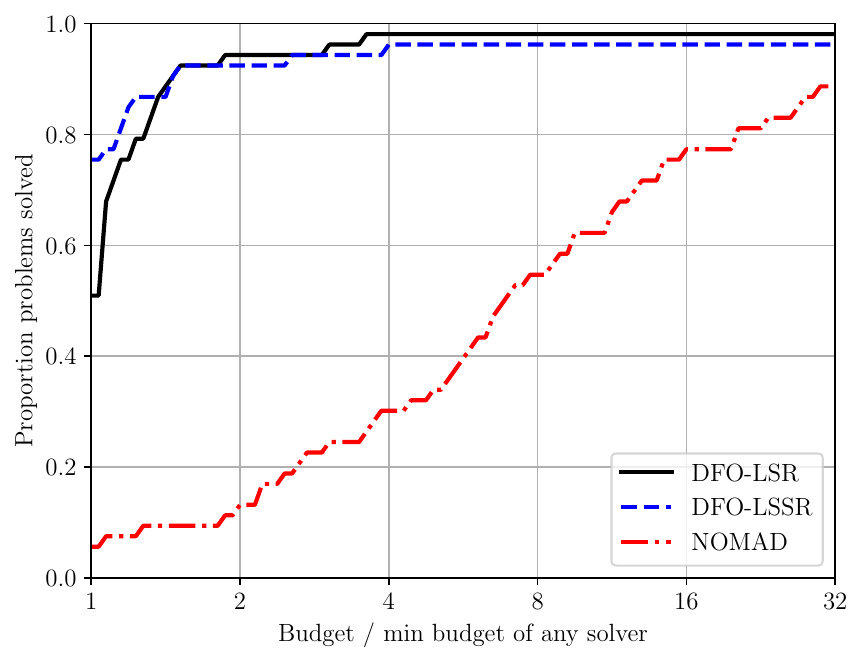}
    \caption{Performance profile, $\tau=10^{-3}$}
  \end{subfigure}
  \\
  \begin{subfigure}[b]{0.48\textwidth}
    \includegraphics[width=\textwidth]{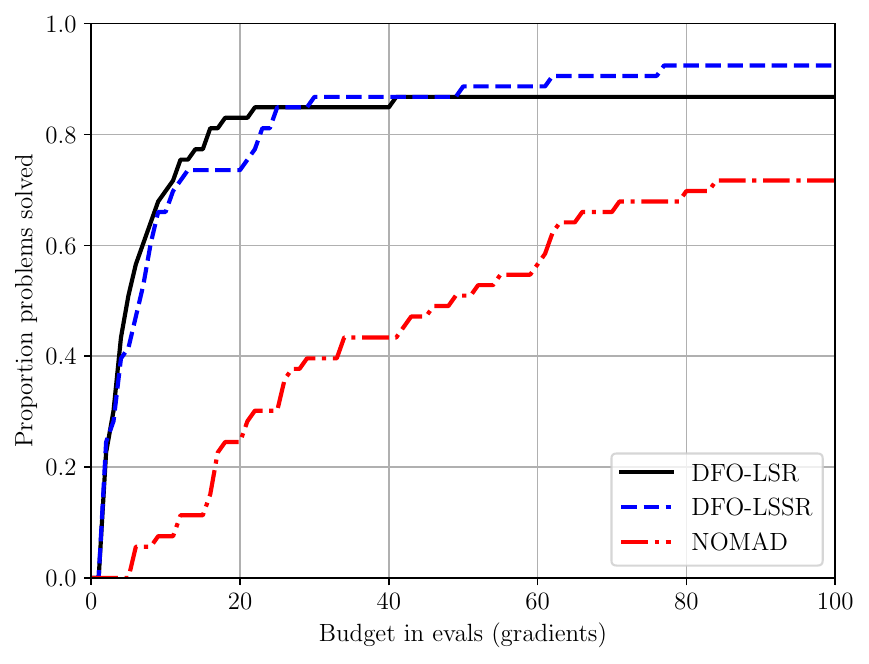}
    \caption{Data profile, $\tau=10^{-5}$}
  \end{subfigure}
  ~
  \begin{subfigure}[b]{0.48\textwidth}
    \includegraphics[width=\textwidth]{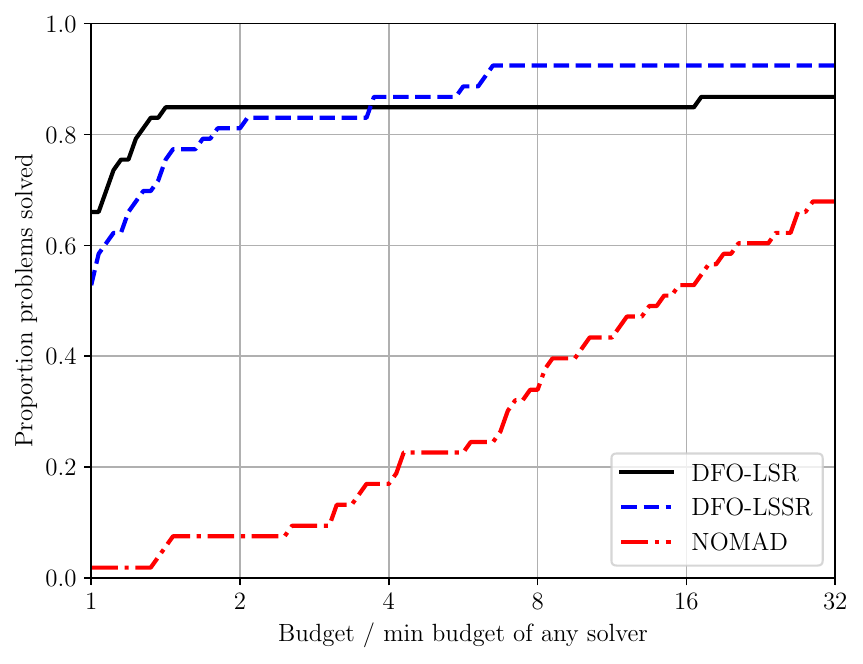}
    \caption{Performance profile, $\tau=10^{-5}$}
  \end{subfigure}
  \\
  \begin{subfigure}[b]{0.48\textwidth}
    \includegraphics[width=\textwidth]{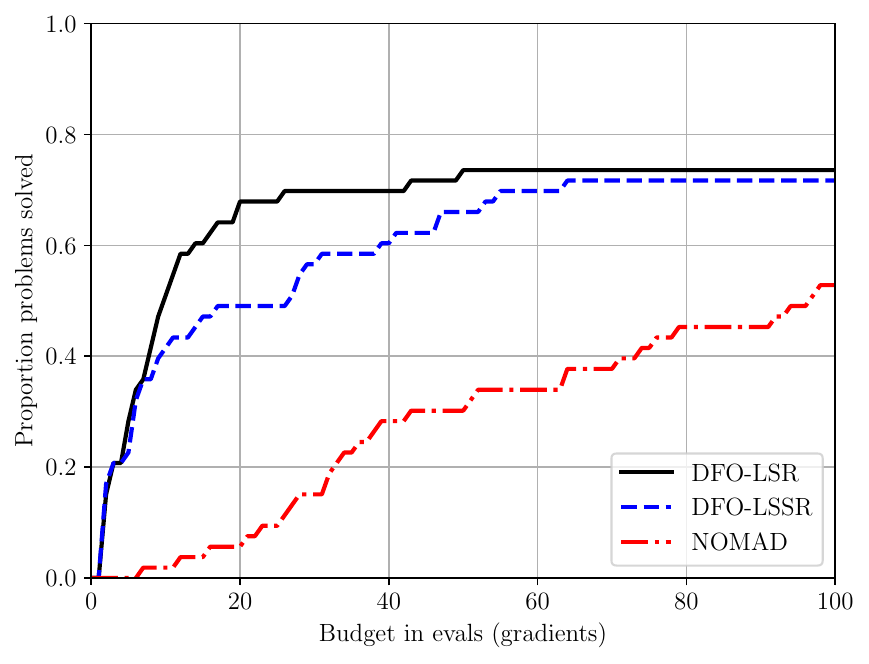}
    \caption{Data profile, $\tau=10^{-7}$}
  \end{subfigure}
  ~
  \begin{subfigure}[b]{0.48\textwidth}
    \includegraphics[width=\textwidth]{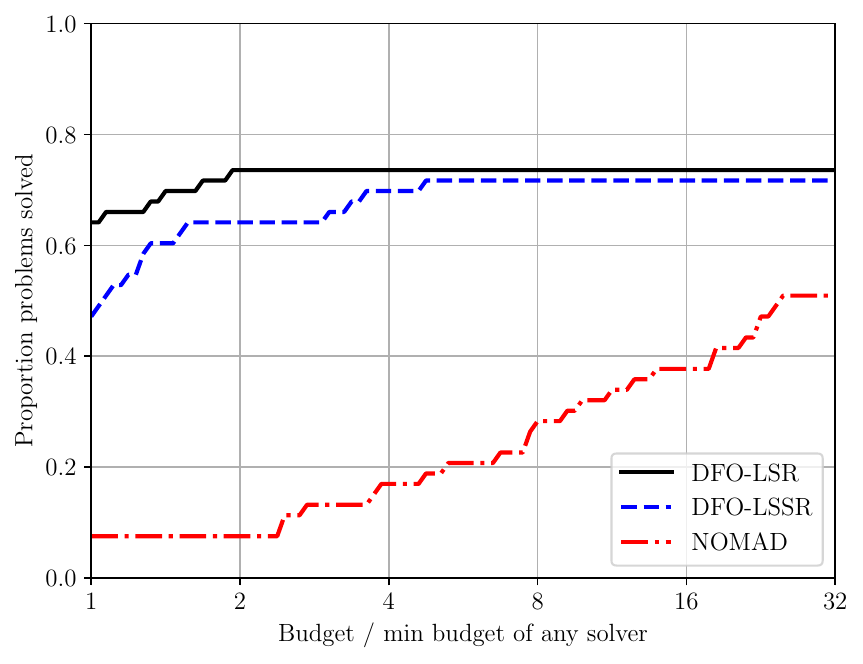}
    \caption{Performance profile, $\tau=10^{-7}$}
  \end{subfigure}
  \caption{Comparison of DFO-LSR, DFO-LSSR and NOMAD on smooth/noiseless test problems for increasing accuracy levels $\tau=10^{-3}$, $\tau=10^{-5}$ and $\tau=10^{-7}$.}
  \label{fig_smooth}
\end{figure}

To compare solvers, we use the data and performance profiles \cite{dolan2002benchmarking}. For every solver $\So$, each problem $p$, we determine the number of function evaluations $N_p(\So;\tau)$ required for a problem to be solved up to a given accuracy level $\tau \in (0,1)$:
\begin{equation} \label{eq_evals_before_solved}
N_p(\So;\tau) := \text{\# objective evals before }\ \Phi(\bx_k) \leq \Phi^* + \tau\left(\Phi(\bx_0)-\Phi^*\right),
\end{equation}
where $\Phi^*$ is an approximation to the true minimum of problem $p$. In our numerical experiments, we take $\Phi^*$ to be the smallest objective value generated by either of the two solvers. We set $N_p(\So;\tau) = \infty$ if the inequality in \eqref{eq_evals_before_solved} is not achieved for a corresponding $\So$, $p$ and $r$ in the maximum computational budget allowed, here $100(n_p+1)$.

We compare solvers by calculating the proportion of problems solved in our problem suite  $\Ps$ for a given computational budget. Consider solver $\So$, problem $p$ and accuracy level $\tau \in (0,1)$, for \emph{data profiles}, we normalize $N_p(\So; \tau)$ by the dimension for problem $p$ and plot 
\begin{equation}
    d_{\So, \tau}(\alpha) \coloneqq \frac{\{p \in \Ps: N_p(\So;\tau) \leq \alpha (n_p + 1)\}}{\abs{\Ps}}, \quad \alpha \in [0, N_g],
\end{equation}
where $N_g$ is the maximum computational budget, measured in simplex gradients. For \emph{performance profiles}, we plot $N_p(\So; \tau)$ normalized by the minimum number of objective evaluations $N_p^*(\tau) \coloneqq \min_{\So} N_p(\So; \tau)$:
\begin{equation}
\pi_{\So,\tau}(\alpha) := \frac{\abs{\{p \in \Ps : N_p(\So;\tau) \leq \alpha N_p^*(\tau)\}}}{\abs{\Ps}}, \quad \alpha \geq 1.
\end{equation}

\subsection{Test Results}

In Figure~\ref{fig_smooth}, we compare our methods DFO-LSR and DFO-LSSR with NOMAD for problems with noiseless objective evaluations, for accuracy levels $\tau\in\{10^{-3}, 10^{-5}, 10^{-7}\}$.
As we would expect, both DFO-LSR and DFO-LSSR significantly outperform NOMAD at all accuracy levels, as they have much more information about the problem structure, including the full residual vector $\br(\bx)$ and the proximal operator of $h$.
Overall, the performance of DFO-LSR and DFO-LSSR are similar, but we observe that DFO-LSSR outperforms DFO-LSR for lower accuracy levels, but DFO-LSR is better able to reach high accuracy solutions, $\tau=10^{-7}$.

For problems with noisy objective evaluations, we show the full numerical results at accuracy levels $\tau\in\{10^{-3},10^{-5}\}$ in Appendix~\ref{app_extra_numerics}.
We do not include the highest accuracy level $\tau=10^{-7}$ here, since such reaching high accuracy requirements is usually impractical for noisy problems without further algorithmic improvements (e.g.~sample averaging).
For noisy problems, again both DFO-LSR and DFO-LSSR outperform NOMAD.
The performance difference between the DFO-LSR and DFO-LSSR is quite small, and so it is not clear from these results which variant is to be preferred in this setting.

\section{Conclusions and Future Work}

We introduce two model-based derivative-free approaches for solving \eqref{eq_prob}, and in particular its specialization to nonlinear least-squares \eqref{eq_prob_nlls}.
First, Algorithm~\ref{alg_dfo} extends the approach from \cite{grapiglia2016derivative} to handle inexact stationarity measures, a simpler sufficient decrease condition for the trust-region subproblem, and a more complicated algorithmic framework inherited from \cite{cartis2019derivative}, without sacrificing global convergence or worst-case complexity results.
Secondly, Algorithm~\ref{alg_dfo_smoothing} adapts the smoothing approach from \cite{garmanjani2016trust} to our setting, and we add to its convergence theory by showing that all accumulation points of the algorithm are Clarke stationary. 
Our numerical results indicate that both approaches perform well, although there is some evidence to suggest that the smoothing approach is slightly better at achieving low accuracy solutions, while Algorithm~\ref{alg_dfo} is better at finding high accuracy solutions.

The most promising direction for future research on this topic is to extend our approaches to the case where $f$ does not have a nonlinear least-squares structure.
The extra difficulty in this setting is to appropriately handle the nonconvex, nonsmooth subproblems that would arise in this setting.

\bibliographystyle{siamplain}
\bibliography{references}

\clearpage

\appendix

\section{Supplementary Numerical Results} \label{app_extra_numerics}
Below we show our numerical results for problems with noisy objectives, described in Section~\ref{sec_test_info}.
Results for problems with multiplicative Gaussian noise are shown in Figure~\ref{fig_multiplicative_gaussian} and problems with additive Gaussian noise in Figure~\ref{fig_additive_gaussian}.

\begin{figure}[H]
  \centering
  \begin{subfigure}[b]{0.48\textwidth}
    \includegraphics[width=\textwidth]{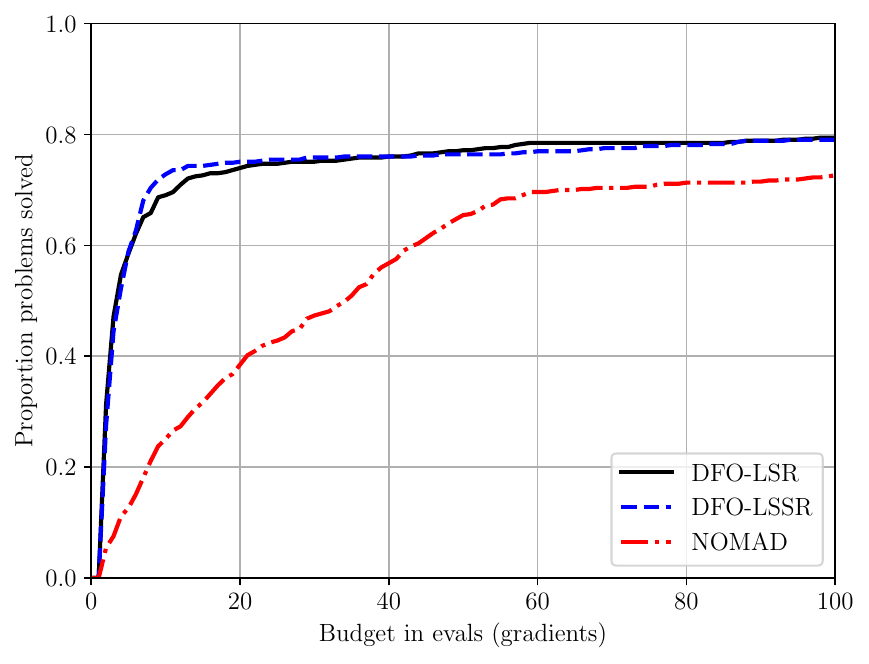}
    \caption{Data profile, $\tau=10^{-3}$}
  \end{subfigure}
  ~
  \begin{subfigure}[b]{0.48\textwidth}
    \includegraphics[width=\textwidth]{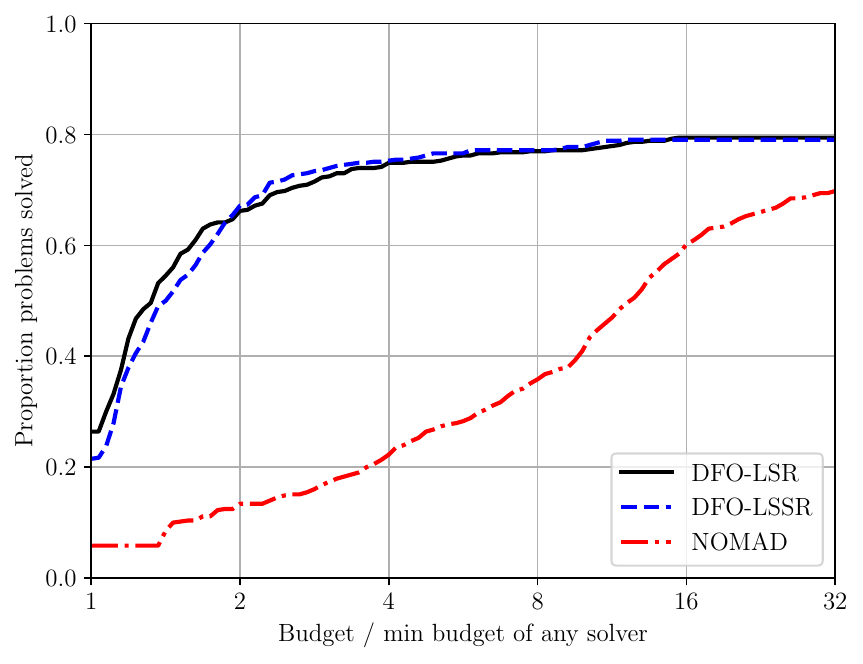}
    \caption{Performance profile, $\tau=10^{-3}$}
  \end{subfigure}
  \\
  \begin{subfigure}[b]{0.48\textwidth}
    \includegraphics[width=\textwidth]{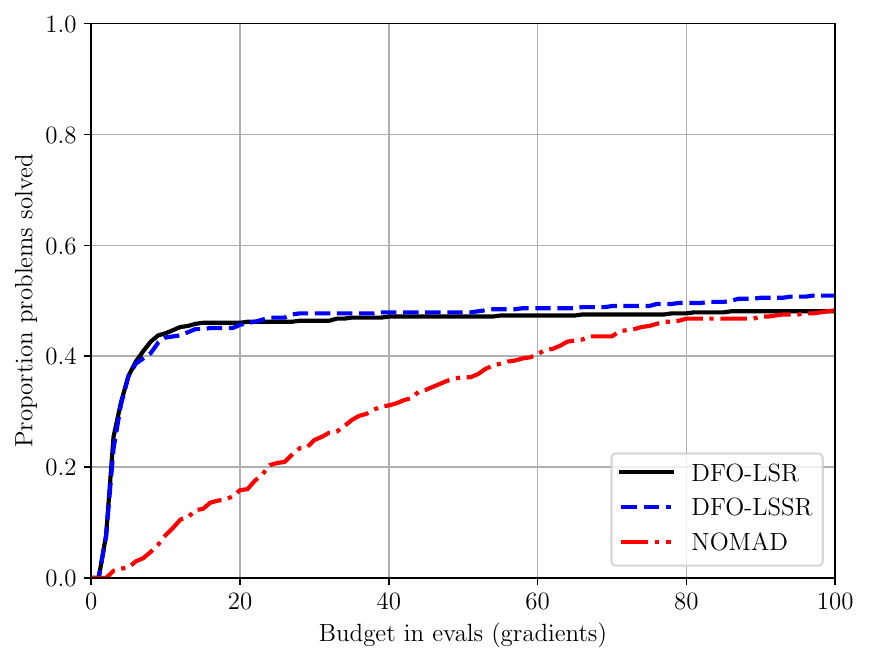}
    \caption{Data profile, $\tau=10^{-5}$}
  \end{subfigure}
  ~
  \begin{subfigure}[b]{0.48\textwidth}
    \includegraphics[width=\textwidth]{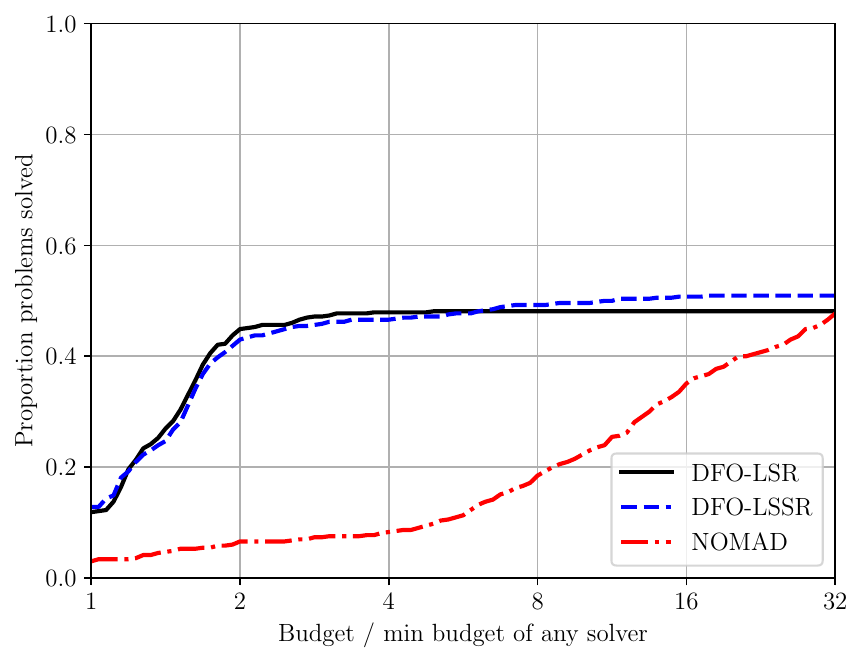}
    \caption{Performance profile, $\tau=10^{-5}$}
  \end{subfigure}
  \caption{Comparison of DFO-LSR, DFO-LSSR and NOMAD on test problems with multiplicative Gaussian noise for accuracy levels $\tau=10^{-3}$ and $\tau=10^{-5}$.}
  \label{fig_multiplicative_gaussian}
\end{figure}

\begin{figure}[H]
  \centering
  \begin{subfigure}[b]{0.48\textwidth}
    \includegraphics[width=\textwidth]{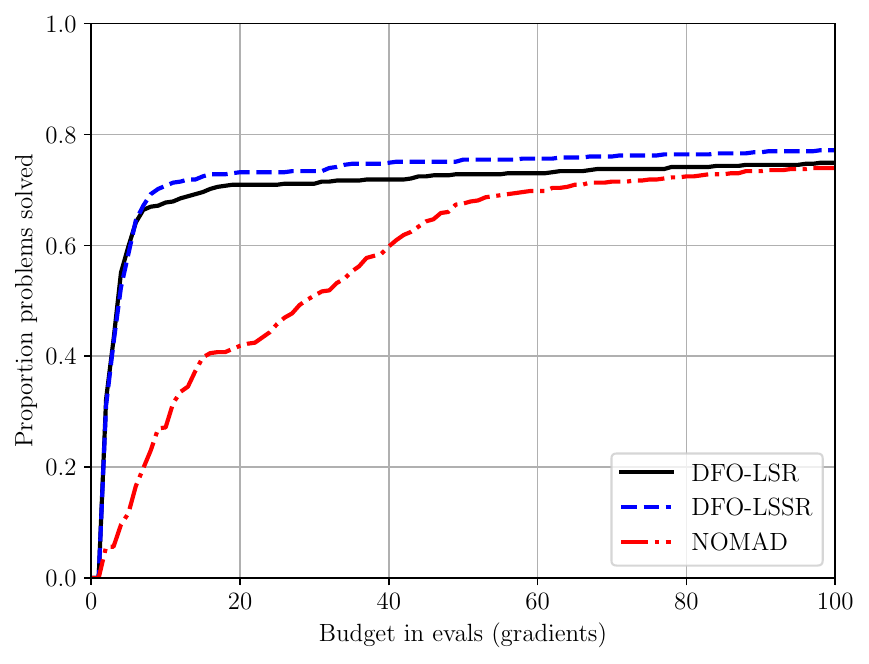}
    \caption{Data profile, $\tau=10^{-3}$}
  \end{subfigure}
  ~
  \begin{subfigure}[b]{0.48\textwidth}
    \includegraphics[width=\textwidth]{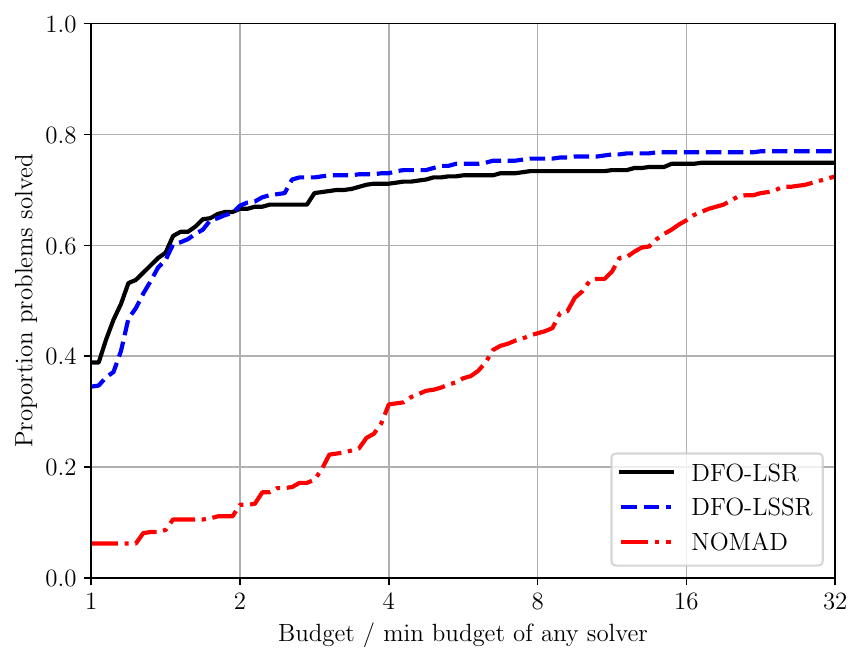}
    \caption{Performance profile, $\tau=10^{-3}$}
  \end{subfigure}
  \\
  \begin{subfigure}[b]{0.48\textwidth}
    \includegraphics[width=\textwidth]{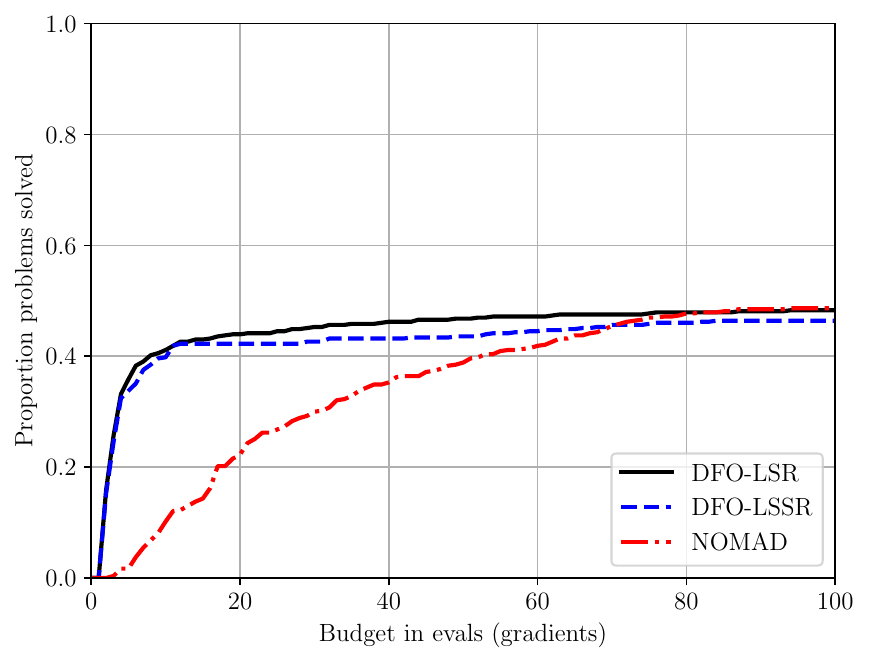}
    \caption{Data profile, $\tau=10^{-5}$}
  \end{subfigure}
  ~
  \begin{subfigure}[b]{0.48\textwidth}
    \includegraphics[width=\textwidth]{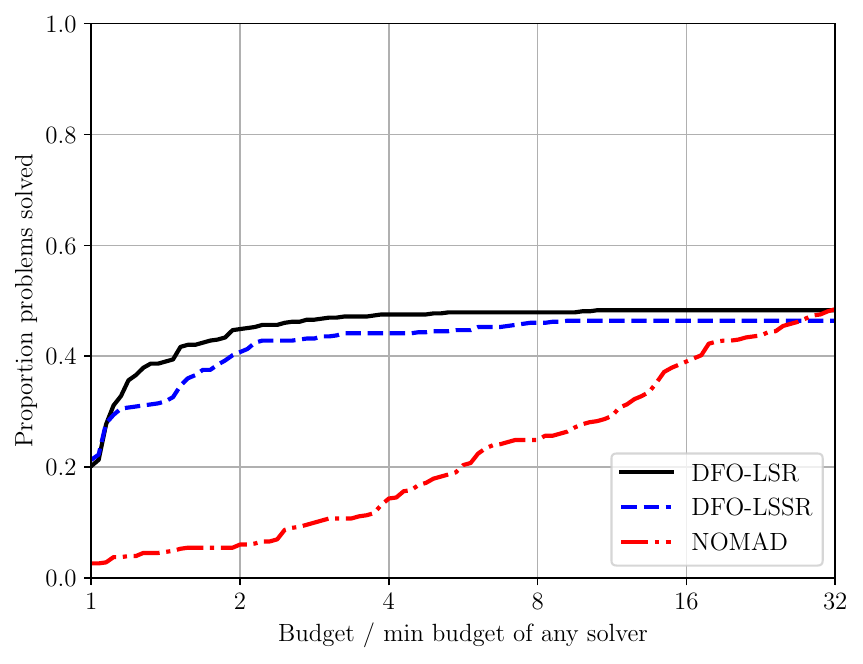}
    \caption{Performance profile, $\tau=10^{-5}$}
  \end{subfigure}
  \caption{Comparison of DFO-LSR, DFO-LSSR and NOMAD on test problems with additive Gaussian noise for accuracy levels $\tau=10^{-3}$ and $\tau=10^{-5}$.}
  \label{fig_additive_gaussian}
\end{figure}

\end{document}